\documentclass[11pt,oneside,reqno]{amsart}
\usepackage[british]{babel}
\usepackage[a4paper]{geometry}
\usepackage{amsmath}
\usepackage{amsthm}
\usepackage{amssymb,stmaryrd,enumerate,bbm,latexsym,xcolor}

\usepackage[pdftex,bookmarks,colorlinks,breaklinks]{hyperref}  
\usepackage[bookmarks,colorlinks,breaklinks]{hyperref}
\definecolor{dullmagenta}{rgb}{0.4,0,0.4}   
\definecolor{darkblue}{rgb}{0,0,0.4}
\hypersetup{linkcolor=blue,citecolor=blue,filecolor=dullmagenta,urlcolor=darkblue} 

\newcommand{\TeXmacs}{T\kern-.1667em\lower.5ex\hbox{E}\kern-.125emX\kern-.1em\lower.5ex\hbox{\textsc{m\kern-.05ema\kern-.125emc\kern-.05ems}}}
\newcommand{\assign}{:=}
\newcommand{\citetexmacs}[1]{This document has been written using GNU {\TeXmacs} \cite{#1}.}
\newcommand{\mathd}{\mathrm{d}}

\newcommand{\tmdfn}[1]{\textbf{#1}}
\newcommand{\tmem}[1]{{\em #1\/}}

\newcommand{\tmnote}[1]{\thanks{\textit{Note:} #1}}
\newcommand{\tmop}[1]{\ensuremath{\operatorname{#1}}}
\newcommand{\tmsep}{, }
\newcommand{\tmstrong}[1]{\textbf{#1}}
\newcommand{\tmtextbf}[1]{\text{{\bfseries{#1}}}}
\newcommand{\tmtextit}[1]{\text{{\itshape{#1}}}}
\newenvironment{enumeratealpha}{\begin{enumerate}[a{\textup{)}}] }{\end{enumerate}}
\newenvironment{enumeratenumeric}{\begin{enumerate}[1.] }{\end{enumerate}}
\newenvironment{enumerateroman}{\begin{enumerate}[i.] }{\end{enumerate}}
\newtheorem{theorem}{Theorem}[section]
\newtheorem{corollary}[theorem]{Corollary}
\newtheorem{definition}[theorem]{Definition}
\newtheorem{lemma}[theorem]{Lemma}
\newtheorem{proposition}[theorem]{Proposition}
{\theoremstyle{remark}\newtheorem{remark}[theorem]{Remark}}

\numberwithin{equation}{section}

\begin{document}

\title{mixing FOR generic rough shear flows}
\tmnote{{\citetexmacs{TeXmacs:website}}}

\author{L.~Galeati}
\address{Batiment 8, Office MA B2455, AMCV Group\\
EPFL Lausanne\\
Switzerland}
\email{lucio.galeati@epfl.ch}

\author{M.~Gubinelli}
\address{Andrew Wiles Building\\ Mathematical Institute\\ University of Oxford}
\email{gubinelli@maths.ox.ac.uk}

\begin{abstract}
  We study mixing and diffusion properties of passive scalars driven by
  \tmtextit{generic} rough shear flows. Genericity is here understood in the
  sense of prevalence and (ir)regularity is measured in the Besov--Nikolskii
  scale $B^{\alpha}_{1, \infty}$, $\alpha \in (0, 1)$. We provide upper and
  lower bounds, showing that in general inviscid mixing in $H^{1 / 2}$ holds
  sharply with rate $r (t) \sim t^{1 / (2 \alpha)}$, while enhanced
  dissipation holds with rate $r (\nu) \sim \nu^{\alpha / (\alpha + 2)}$. Our
  results in the inviscid mixing case rely on the concept of
  $\rho$-irregularity, first introduced by Catellier and Gubinelli
  (Stoc.~Proc.~Appl.~126, 2016) and provide some new insights compared to the
  behavior predicted by Colombo, Coti~Zelati and Widmayer (Ars Inven. Anal., 2021).
  
  \
  
  {\noindent}\tmtextbf{MSC(2020):} 35Q35, 37C20, 76F25, 76R50.
\end{abstract}


{\maketitle}

{\tableofcontents}

\section{Introduction}\label{sec1}

We are interested in the long time behavior of solutions $f$ to
\begin{equation}
  \left\{\begin{array}{l}
    \partial_t f + u \partial_x f = \nu \Delta f\\
    f|_{t = 0} = f_0, \quad \int_{\mathbb{T}} f_0 (x, y) \mathd x = 0
  \end{array}\right. \label{eq:shear-flow}
\end{equation}
on the $2$-dimensional flat torus $\mathbb{T}^2$. The
PDE~\eqref{eq:shear-flow} is an advection-diffusion equation associated to a
shear flow $u = u (y) : \mathbb{T} \rightarrow \mathbb{R}$, $f :
\mathbb{R}_{\geqslant 0} \times \mathbb{T}^2 \rightarrow \mathbb{R}$ with
initial condition $f_0 \in L^2 (\mathbb{T}^2)$ and where $\nu \in [0, 1]$ is
the diffusion coefficient. Defining $\bar{u} : \mathbb{T}^2 \rightarrow
\mathbb{R}^2$ as $\bar{u} (x, y) : = (u (y), 0)^T$,
equation~\eqref{eq:shear-flow} may be written as
\begin{equation}
  \partial_t f + \bar{u} \cdot \nabla f = \nu \Delta f
  \label{eq:passive-scalar}
\end{equation}
which is the equation for a passive scalar $f$ advected by the velocity field
$\bar{u}$. Note that $\bar{u}$ is a divergence-free vector field and a
stationary solution to $2$D Euler equations.

Exactly for this reason, shear flows have received a lot of attention in the
literature, in connection to the problem of understanding the interaction
between mixing and diffusion in fluid mechanics and the transfer of energy
from large to small scales for the scalar $f$. In particular, shear flows are
sufficiently simple to allow explicit calculations, while presenting a highly
non trivial behavior, as already observed by Kelvin in~{\cite{kelvin}} in the
case of the Couette flow $u (y) = y$.

Observe that for continuous $u$, eq.~\eqref{eq:shear-flow} can be solved
explicitly by Feynman--Kac formula, giving
\begin{equation}
  f_t (x, y) =\mathbb{E} \left[ f_0 \left( x - \int_0^t u \left( y + \sqrt{2
  \nu} B^2_s \right) \mathd s + \sqrt{2 \nu} B^1_t, y + \sqrt{2 \nu} B^2_t
  \right) \right] \label{eq:dissipative-solution}
\end{equation}
where $B = (B^1, B^2)$ is a standard 2D Brownian motion~(Bm). In the case $\nu
= 0$ we obtain
\begin{equation}
  f_t (x, y) = f_0 (x - t u (y), y) . \label{eq:inviscid-solution}
\end{equation}
Both formulas~\eqref{eq:dissipative-solution} and~\eqref{eq:inviscid-solution}
can then be extended to the case $u \in L^1 (\mathbb{T})$, \footnote{For $u
\in L^1 (\mathbb{T})$, the formal expression $\int_0^t u \left( y + \sqrt{2
\nu} B^2_s \right) \mathd s$ in~\eqref{eq:dissipative-solution} can be made
rigorous using the local time of $B^2$; alternatively,
equation~\eqref{eq:shear-flow} can be solved by applying the Fourier transform
in the $x$-variable and solving the family of equations for $f^k = P_k f$, see
the beginning of Appendix~\ref{app:wei-extension} for more details.} in which
case eq.~\eqref{eq:shear-flow} must be understood in the weak sense, and
generate continuous semigroups $e^{t (- u \partial_x + \nu \Delta)}$ on $L^2
(\mathbb{T}^2)$. Yet, they do not provide any immediate insight on the long
time behavior of the solution $f$, in particular on the decay in time of
quantities like $\| f_t \|_{H^{- s}}$ and $\| f_t \|_{L^2}$.

\

Following the line of research initiated
in~{\cite{wei}},~{\cite{colombozelati}}, we consider rough shear flows, in the
sense of requiring $u \in B^{\alpha}_{1, \infty} (\mathbb{T})$ for some
$\alpha \in (0, 1)$. Here $B^{\alpha}_{1, \infty} (\mathbb{T})$ denote the
Besov--Nikolskii spaces, see Appendix~\ref{app:besov} for their definition.

We are interested in understanding the behavior of \tmtextit{generic} $u \in
B^{\alpha}_{1, \infty} (\mathbb{T})$, a problem explicitly left open
in~{\cite{colombozelati}}. For this purpose we adopt the measure-theoretic
notion of genericity given by the theory of prevalence, developed by Hunt,
Sauer and Yorke~{\cite{huntsauer}} to provide an analogous of ``Lebesgue
almost every'' on infinite dimensional spaces, see
Section~\ref{sec:prevalence} for more details. In what follows the expression
``for almost every $\varphi \in E$'', where $E$ is a function space, is
understood in the sense of prevalence.

The next statement summarizes our main findings.

\begin{theorem}
  \label{thm:main-thm-1}Let $\alpha \in (0, 1)$. The following hold:
  \begin{enumerateroman}
    \item For almost every $u \in B^{\alpha}_{1, \infty} (\mathbb{T})$ we have
    inviscid mixing in the scale $H^{1 / 2} (\mathbb{T}^2)$, in the following
    sense: for any $\tilde{\alpha} > \alpha$, there exists $C = C (\alpha,
    \tilde{\alpha}, u)$ such that, for any $f_0 \in H^{1 / 2} (\mathbb{T})$
    satisfying $\int_{\mathbb{T}} f \left( x, \cdot \, \right) \mathd x \equiv
    0$, it holds
    \[ \| e^{- t u \partial_x} f_0 \|_{H^{- 1 / 2}} \leqslant C t^{-
       \frac{1}{2 \tilde{\alpha}}}  \| f_0 \|_{H^{1 / 2}} \quad \forall \, t
       \geqslant 0. \]
    \item For almost every $u \in B^{\alpha}_{1, \infty} (\mathbb{T})$ we have
    enhanced dissipation in the following sense that: for any $\tilde{\alpha}
    > \alpha$ there exist $C_i = C (\alpha, \tilde{\alpha}, u)$ such that, for
    any $f_0 \in L^2 (\mathbb{T})$ satisfying $\int_{\mathbb{T}} f \left( x,
    \cdot \, \right) \mathd x \equiv 0$, it holds
    \[ \| e^{t (- u \partial_x + \Delta)} f_0 \|_{L^2} \leqslant C_1 \exp
       \left( - C_2 t \nu^{\frac{\tilde{\alpha}}{\tilde{\alpha} + 2}} \right)
       \| f_0 \|_{L^2} \quad \forall \, t \geqslant 0, \nu \in [0, 1] . \]
  \end{enumerateroman}
\end{theorem}

In the above statement, the condition $\int_{\mathbb{T}} f \left( x, \cdot \,
\right) \mathd x \equiv 0$ is necessary, as it naturally ensures that $f$
witnesses the effect of the transport operator $u \partial_x$; indeed $g_t (y)
: = \int_{\mathbb{T}} f_t (x, y) \mathd x$ must solve the standard heat
equation $\partial_t g = \nu \partial_y^2 g$ and thus cannot exhibit any
mixing/enhanced dissipation effect.

There is no obvious a priori reason to work with the spaces $B^{\alpha}_{1,
\infty} (\mathbb{T})$ (e.g. in~{\cite{colombozelati}} the authors deal with
$C^{\alpha} (\mathbb{T}) = B^{\alpha}_{\infty, \infty} (\mathbb{T})$), rather
they arise naturally in our analysis. One of the main intuitions of the
present paper is the identification of such spaces as the correct one for
studying generic inviscid mixing and enhanced dissipation properties of shear
flows. At the same time, let us mention that the only truly relevant parameter
is $\alpha \in (0, 1)$: indeed statements similar to those of
Theorem~\ref{thm:main-thm-1} can be given for the (smaller) spaces
$B^{\alpha}_{p, q} (\mathbb{T})$ for any choice of $p, q \in [1, \infty]$, see
Remark~\ref{rem:intro} below.

\

Before moving further, let us heuristically motivate the connection between
Points~\tmtextit{i.} and~\tmtextit{ii.} of Theorem~\ref{thm:main-thm-1} and
why it is natural to expect $\nu^{\alpha / (\alpha + 2)}$ to appear, given the
decay $\| f_t \|_{H^{- 1 / 2}} \lesssim t^{- 1 / (2 \alpha)}$. In fact, the
argument can be given in a much more general framework: let $f^{\nu}$ be a
solution to~\eqref{eq:passive-scalar} with $\nu > 0$, $\int_{\mathbb{T}^d} f_0
(z) \mathd z = 0$ and $\bar{u} : \mathbb{T}^d \rightarrow \mathbb{R}^d$ be a
divergence free vector field; then $f^{\nu}$ satisfies the energy balance
\[ \frac{\mathd}{\mathd t} \| f^{\nu}_t \|_{L^2}^2 = - 2 \nu \| \nabla
   f^{\nu}_t \|_{L^2}^2 . \]
Now assume the solution $f$ to the transport equation $\partial_t f + \bar{u}
\cdot \nabla f = 0$ to satisfy the decay $\| f_t \|_{\dot{H}^{- s}} \lesssim
t^{- s / \alpha}$ for suitable parameters $\alpha > 0, s \in
(0, 1]$ (for $s > 1$, one may reduce to $s = 1$ by Riesz--Thorin interpolation
theorem). For $\nu \ll 1$ and sufficiently short times, we expect $f^{\nu}$
and $f$ to stay close and therefore $f^{\nu}$ to exhibit the same decay as
$f$. By the interpolation inequality

\begin{align*}
  \| f \|_{L^2} \lesssim & \| f \|_{\dot{H}^{- s}}^{\frac{1}{1 + s}}  \|
  \nabla f \|_{L^2}^{\frac{s}{s + 1}},
\end{align*}

we deduce that
\begin{equation}
  \frac{\mathd}{\mathd t} \| f^{\nu}_t \|_{L^2}^{- \frac{2}{s}} \sim \nu \|
  f^{\nu}_t \|_{L^2}^{- 2 \left( \frac{s + 1}{s} \right)} \| \nabla f^{\nu}_t
  \|_{L^2}^2 \gtrsim \nu \| f^{\nu}_t \|_{\dot{H}^{- s}}^{-
  \frac{2}{s}} \gtrsim \nu t^{\frac{2}{\alpha}} .
  \label{eq:relation-inviscid-diss}
\end{equation}
Assume for simplicity $\| f_0 \|_{L^2} = 1$ and define $\tau > 0$ to be the
first time such that $\| f^{\nu}_t \|_{L^2} = 1 / 2$.
Integrating~\eqref{eq:relation-inviscid-diss} over $[0, \tau]$ we obtain

\begin{align*}
  1 \sim 2^{\frac{2}{s}} - 1 & \gtrsim \, \nu \int_0^{\tau}
  t^{\frac{2}{\alpha}} \mathd t \sim \nu \tau^{1 + \frac{2}{\alpha}} = \left(
  \nu^{\frac{\alpha}{\alpha + 2}} \tau \right)^{\frac{\alpha}{\alpha + 2}} .
\end{align*}

Namely, in order for the energy $\| f^{\nu}_t \|_{L^2}$ to be reduced by half
by the dynamics, we need to wait for at most $\tau \lesssim \nu^{- \alpha /
(\alpha + 2)}$. Iterating the argument on intervals $[n \tau, n (\tau + 1)]$
would then produce an asymptotic decay at least of the form $\exp (- C t
\nu^{\frac{\alpha}{\alpha + 2}})$.

While the argument is clearly heuristic, it predicts the correct exponent
$\frac{\alpha}{\alpha + 2}$ and works for any choice of the parameter $s > 0$
(in particular for $s = 1 / 2$ as in Theorem~\ref{thm:main-thm-1}) and not
only for $s = 1$, which is the case receiving the most attention in the
literature.

Unfortunately, there are only few rigorous quantitative results connecting
explicitly inviscid mixing and enhanced dissipation properties
(see~{\cite{zelati2020relation}} and the references therein) and they appear
not to be optimal. For instance for $s \in (0, 1]$, an application
of~Corollary~2.3 from~{\cite{zelati2020relation}} would only predict a decay
\[ \| f_t \|_{L^2} \leqslant \exp (- C \nu^{q_s} t) \| f_0 \|_{L^2}, \quad q_s
   \assign \frac{\alpha (1 + s)}{\alpha + s + \alpha s} ; \]
in particular $q_1 = \frac{2 \alpha}{2 \alpha + 1}$ while $q_{1 / 2} = \frac{3
\alpha}{3 \alpha + 1}$.

\

\tmtextbf{Relation with existing literature.} Understanding the interaction
between mixing and diffusion is one of the most fundamental problems in fluid
mechanics, dating back to the works of Kelvin~{\cite{kelvin}} and
Reynolds~{\cite{reynolds}}.

In the pioneering work~{\cite{constantin2008}}, such relation has been
formalized mathematically by introducing the concept of \tmtextit{relaxation
enhancing flows}; the result has been recently revisited in a more
quantitative fashion in the works~{\cite{zelati2020relation,fengiyer}}. The
use of weak norms $H^{- s}$ in order to quantify mixing of passive scalars was
first introduced in~{\cite{thiffeault}}.

Shear flows and circular flows in particular have been recently studied by
several authors, employing a variety of technique, including stationary phase
methods and hypocoercivity
schemes~{\cite{bedrossian2017enhanced,zelati2020stable,zelati2020separation}},
spectral methods~{\cite{wei}} and stochastic analysis~{\cite{zelatidrivas}}.
Roughly speaking, the main known results for~\eqref{eq:shear-flow} are the
following:
\begin{itemize}
  \item If $u \in C^{n + 1}$ has a finite number of critical points with
  maximal order $n$, then enhanced dissipation holds with $r (\nu) \sim
  \nu^{\frac{n}{n + 2}} (1 + \log \nu^{- 1})^{- 1}$, see Theorem~1.1
  in~{\cite{bedrossian2017enhanced}}.
  
  \item There exist $u \in C^{\alpha}$, $\alpha \in (0, 1)$, for which
  enhanced dissipation holds with $r (\nu) \sim \nu^{\frac{\alpha}{\alpha +
  2}}$, see Theorem~5.1 from~{\cite{wei}}.
  
  \item The above results are sharp, up to logarithmic corrections, in the
  sense that for $u \in C^{n + 1}$ (resp. $u \in C^{\alpha}$) the best
  possible rate is $r (\nu) \sim \nu^{\frac{n}{n + 2}}$ (resp. $r (\nu) \sim
  \nu^{\frac{\alpha}{\alpha + 2}}$), see Theorem~4 in~{\cite{zelatidrivas}};
  the proof is based on the Lagrangian Fluctuation Dissipation relation
  introduced in~{\cite{drivas1}},~{\cite{drivas2}}.
\end{itemize}
Let us also mention the remarkable stable mixing estimate obtained
in~{\cite{zelati2020stable}} for $u$ satisfying Assumption~(H) therein.
Motivated by the above results, the authors of~{\cite{colombozelati}} explore
the mixing and enhanced dissipation properties of rough shear flows, namely
$u$ sharply $\alpha$--H\"{o}lder for $\alpha \in (0, 1)$. In particular, they
construct a Weierstrass-type flow $u$ such that the following hold (see
Theorem~1.1 in~{\cite{colombozelati}}):
\begin{enumeratenumeric}
  \item enhanced dissipation holds with rate $r (\nu) \sim
  \nu^{\frac{\alpha}{\alpha + 2}}$, confirming the results from~{\cite{wei}};
  
  \item along suitable sequences $t_n \rightarrow \infty$, inviscid mixing
  holds on $H^1$ with rate $r (t) \sim t^{1 / \alpha}$:
  \[ \| e^{- t_n u \partial_x} f_0 \|_{H^{- 1}} \lesssim t_n^{-
     \frac{1}{\alpha}}  \| f_0 \|_{H^1} . \]
  \item however, to the authors' surprise, there exist other sequences
  $\tilde{t}_n \rightarrow \infty$ on which inviscid mixing only holds with
  rate $r (t) \sim t$, in the sense that
  \[ \| e^{- \tilde{t}_n u \partial_x} f_0 \|_{H^{- 1}} \gtrsim
     \tilde{t}_n^{- 1}  \| f_0 \|_{H^1} . \]
\end{enumeratenumeric}
In particular, the inviscid mixing rate $r (t) \sim t$ is the same attained by
suitable Lipschitz functions; the authors wonder whether such a discrepancy
between Points~2. and~3. is to be expected for generic flows $u \in
C^{\alpha}$, see the paragraph ``Perspectives'', p.3
in~{\cite{colombozelati}}.

The main aim of the present work is to give a negative answer to the above
question, while letting a more natural picture emerge in the context of
generic rough shear flows. Theorem~\ref{thm:main-thm-1} shows that for generic
$u \in B^{\alpha}_{1, \infty}$ (similarly for $u \in C^{\alpha}$, see
Remark~\ref{rem:intro}) inviscid mixing holds on $H^{1 / 2}$
with rate $r (t) \sim t^{1 / 2 \alpha}$, uniformly over all $t \geqslant 0$.
Such a decay is also the best possible, see
Theorem~\ref{thm:main-thm-inviscid} below. On the other hand,
Theorem~\ref{thm:main-thm-1} confirms the enhanced dissipation rate $r (\nu)
\sim \nu^{\alpha / (\alpha + 2)}$, already identified
in~{\cite{wei,colombozelati}}, as a property of generic shear flows.

We believe that the use of less standard spaces $B^{\alpha}_{1, \infty}$ and
mixing norms $H^{- s}$ with $s \neq 1$ to be some of the main contributions of
this work, compared to previous literature; they arise naturally in
computations, rather than being a mathematical artifact. A complete picture is
however still missing; for instance, the question whether generic $u \in
B^{\alpha}_{1, \infty}$ satisfy inviscid mixing on $H^1$ with
rate $r (t) \sim t^{1 / \alpha}$ is still open and goes beyond
our current methods.

\

\tmtextbf{Structure of the proof.} As done frequently in the literature, in
order to prove Theorem~\ref{thm:main-thm-1} for the PDE~\eqref{eq:shear-flow},
we will pass to study its hypoelliptic counterpart
\begin{equation}
  \partial_t f + u \partial_x f = \nu \partial_y^2 f
  \label{eq:hypoelliptic-shear-flow}
\end{equation}
again under the assumption $\int_{\mathbb{T}} f_0 (x, y) \mathd x = 0$ for all
$y \in \mathbb{T}$.

For $k \in \mathbb{Z}_0 \assign \mathbb{Z} \setminus \{ 0 \}$, define the
Fourier transform in the $x$-variable as
\[ (P_k f) (y) \assign \int_{\mathbb{T}} f (x, y) e^{- i k x} \mathd x \]
so that any $f : \mathbb{T}^2 \rightarrow \mathbb{R}$ has a decomposition $f
(x, y) = \sum_k (P_k f) (y) e^{i k x}$. If $f$
solves~\eqref{eq:hypoelliptic-shear-flow}, then for each $k \in \mathbb{Z}_0$
the function $f^k_t \assign P_k f_t$ solves the one dimensional complex valued
PDE (harmonic oscillator)
\begin{equation}
  \partial_t f^k + i k u f^k = \nu \partial_y^2 f^k .
  \label{eq:harmonic-oscillator}
\end{equation}
For $k \in \mathbb{Z}_0$, $\nu \geqslant 0$ and $u \in L^1 (\mathbb{T})$, the
PDE~\eqref{eq:harmonic-oscillator} has an associated semigroup on $L^2
(\mathbb{T}; \mathbb{C})$, which we denote by $e^{t (- i k u + \nu
\partial_y^2)}$; observe that the parameter $k$, up to its sign, may be
removed by the rescaling $\tilde{t} = t | k |$, $\tilde{\nu} = \nu / | k |$.
In this way the study of asymptotic behavior of $f^k$ may be reduced to that
of $f^{\pm 1}$, which motivates the following definitions.

Note that whenever we refer to a \tmtextit{rate} $r : \mathbb{R}_{\geqslant
0} \rightarrow \mathbb{R}_{\geqslant 0}$, we always assume it to be a
continuous, increasing function.

\begin{definition}
  \label{def:mixing}A velocity field $u \in L^1 (\mathbb{T})$ is said to be
  {\tmdfn{mixing}} on the scale $H^s (\mathbb{T}; \mathbb{C})$, $s \geqslant
  0$, with rate $r_{s \text{-mix}}$, if there exist a constant
  $C > 0$ such that
  \begin{equation}
    \| e^{- i t k u} \|_{H^s \rightarrow H^{- s}} \leqslant
    \frac{C}{r_{s \text{-mix}} (t | k |)}  \quad \forall \: k
    \in \mathbb{Z}_0, \; t \geqslant 1. \label{eq:def-mixing}
  \end{equation}
\end{definition}

\begin{definition}
  \label{def:diffusion-enhancing}A velocity field $u \in L^1 (\mathbb{T})$ is
  said to be {\tmdfn{diffusion enhancing}} on $L^2 (\mathbb{T}; \mathbb{C})$
  with rate $r_{\text{dif}}$ if there exists a constant $C > 0$
  such that
  \begin{equation}
    \| e^{t (- i k u + \nu \partial_y^2)} \|_{L^2 \rightarrow L^2} \leqslant C
    \exp \left( - r_{\text{dif}} \left( \frac{\nu}{| k |}
    \right) | k | t \right)  \quad \forall \, k \in \mathbb{Z}_0, \; \nu \in
    (0, 1], \; t \geqslant 1. \label{eq:def-enhancing}
  \end{equation}
\end{definition}

The following theorems, which are the main results of the paper, provide sharp
inviscid mixing and enhanced diffusion statements for generic shear flows. In
particular, they describe precisely the behavior of solutions
to~\eqref{eq:shear-flow} at each Fourier level $P_k$.

\begin{theorem}[Inviscid case $\nu = 0$]
  \label{thm:main-thm-inviscid}Let $\alpha \in (0, 1)$.
  \begin{enumeratealpha}
    \item Lower bound. Suppose that $u \in B^{\alpha}_{1, \infty}
    (\mathbb{T})$ is mixing on the scale $H^{1 / 2} (\mathbb{T}; \mathbb{C})$
    with rate $r_{1 / 2 \text{-mix}}$, in the sense of
    Definition~\ref{def:mixing}; then necessarily $r_{1 / 2
    \text{-mix}} (t) \lesssim t^{\frac{1}{2 \alpha}}$.
    
    \item Upper bound. Almost every $u \in B^{\alpha}_{1, \infty}
    (\mathbb{T})$ satisfies the following property: for any $\tilde{\alpha} >
    \alpha$, $u$ is mixing on the scale $H^{1 / 2} (\mathbb{T}; \mathbb{C})$
    with rate $r_{1 / 2 \text{-mix}} (t) \gtrsim t^{\frac{1}{2
    \tilde{\alpha}}}$.
  \end{enumeratealpha}
\end{theorem}

\begin{theorem}[Dissipative case $\nu > 0$]
  \label{thm:main-thm-enhanced}Let $\alpha \in (0, 1)$.
  \begin{enumeratealpha}
    \item Lower bound. Suppose that $u \in B^{\alpha}_{1, \infty}
    (\mathbb{T})$ is diffusion enhancing with rate
    $r_{\text{dif}}$, in the sense of
    Definition~\ref{def:diffusion-enhancing}; then necessarily
    $r_{\text{dif}} (\nu) \lesssim \nu^{\frac{\alpha}{\alpha +
    2}}$.
    
    \item Upper bound. Almost every $u \in B^{\alpha}_{1, \infty}
    (\mathbb{T})$ satisfies the following property: for any $\tilde{\alpha} >
    \alpha$, $u$ is diffusion enhancing with rate
    $r_{\text{dif}} (\nu) \gtrsim \nu^{\tilde{\alpha} /
    (\tilde{\alpha} + 2)}$.
  \end{enumeratealpha}
\end{theorem}

Theorems~\ref{thm:main-thm-inviscid} and~\ref{thm:main-thm-enhanced} will be
proven respectively in Sections~\ref{sec:inviscid-mixing}
and~\ref{sec:enhanced-dissipation}, which are structured in a very similar
way. Roughly speaking, the strategy we adopt in proving upper and lower bounds
may be summarized in three main steps:
\begin{enumeratenumeric}
  \item In both cases, the lower bound follows from estimates which explicitly
  employ the regularity assumption $u \in B^{\alpha}_{1, \infty}$; in the case
  $\nu > 0$, we need to preliminary establish a Lagrangian
  Fluctuation-Dissipation relation for the PDE~\eqref{eq:harmonic-oscillator}
  \ (see Proposition~\ref{prop:FDR}) similarly in spirit to what was done
  in~{\cite{zelatidrivas}}.
  
  \item The upper bound is satisfied by any $u$ enjoying a suitable analytic
  property, which encodes its \tmtextit{irregularity}. It turns out that the
  right properties are given respectively by $\rho$-irregularity (see
  Definition~\ref{def:rho-irr}) for $\nu = 0$ and by Wei's irregularity
  condition (see Definition~\ref{def:wei-condition}) for $\nu > 0$. A shear
  flow $u$ satisfying any of such properties necessarily enjoys only limited
  regularity in the scales $B^{\alpha}_{1, \infty}$ (see
  Proposition~\ref{prop:rho-irr-besov-nikolskii} and
  Lemma~\ref{lem:irr-wei}), confirming that these are the correct spaces to
  work with.
  
  \item Finally, we show that a.e. $u \in B^{\alpha}_{1, \infty}$ is
  $\rho$-irregular (resp. satisfies Wei's condition), see
  Section~\ref{sec:prevalence-mixing} (resp.
  Section~\ref{sec:prevalence-enhanced}). This is achieved by probabilistic
  methods, using the law of fractional Brownian motions (see
  Section~\ref{sec:fbm} for details) to construct a measure witnessing the
  prevalence of such properties.
\end{enumeratenumeric}
\begin{remark}
  \label{rem:intro}Let us stress that points~\tmtextit{a}) of
  Theorems~\ref{thm:main-thm-inviscid}-\ref{thm:main-thm-enhanced} hold for
  \tmtextit{all} $u \in B^{\alpha}_{1, \infty}$, not only generic elements.
  Since $\mathbb{T}$ is finite, we have the embeddings $B^{\alpha}_{p, q}
  \hookrightarrow B^{\alpha}_{1, \infty}$ for any $p, q \in [1, \infty]$, thus
  the lower bound is true for all $u \in B^{\alpha}_{p, q}$ as well. On the
  other hand, the proofs of points~\tmtextit{b}) of
  Theorems~\ref{thm:main-thm-inviscid}-\ref{thm:main-thm-enhanced} can be
  easily readapted to provide the same statements for almost every $u \in
  B^{\alpha}_{p, q}$, for any choice of $p, q \in [1, \infty]$.
  
  In particular, one could always work with the spaces $C^{\alpha} =
  B^{\alpha}_{\infty, \infty}$ if desired. There are however several reasons
  for working with $B^{\alpha}_{1, \infty}$ or more generally $B^{\alpha}_{p,
  q}$ instead of $C^{\alpha}$.
  
  Mathematically, such spaces include genuinely discontinuous functions, as
  well as (possibly continuous) functions of finite $p$-variation for any $p
  \in [1, \infty]$: it holds
  \[ B^{1 / p}_{p, 1} \hookrightarrow V^p_c \hookrightarrow V^p
     \hookrightarrow B^{1 / p}_{p, \infty}, \]
  see Proposition 4.3 from~{\cite{liu2020}}, Proposition 2.3
  from~{\cite{friz2021}} for more details.
  
  Physically, a simple way to explain singularities in fully developed
  turbulence is by means of structure functions (see
  e.g.~{\cite{frisch1980}}), which are closely related to the finite
  difference characterization of Besov spaces $B^{\alpha}_{p, \infty}$.
  Turbulence is also believed to be closely connected to multifractality
  (again we refer to the appendix of~{\cite{frisch1980}}), a feature which is
  absent from generic $u \in C^{\alpha}$ (which are monofractal) but instead
  manifested by almost every $u \in B^{\alpha}_{p, q}$,
  see~{\cite{jaffard2000,fraysse2006,fraysse2010}}.
  
  Our results show that the only relevant parameter in understanding mixing
  and enhanced dissipation rates for a.e. $u \in B^{\alpha}_{p, q}$ is $\alpha
  \in (0, 1)$, regardless of the values of $p, q$; thus there is no apparent
  connection between mixing and multifractal features of $u$, at
  least in the setting of shear flows.
\end{remark}

\tmtextbf{Structure of the paper.} In Section~\ref{sec:preliminaries} we
shortly recall some of the main tools we will be working with, specifically
the theory of prevalence and a relevant class of Gaussian
processes, which includes fractional Brownian motion.

Sections~\ref{sec:inviscid-mixing} and~\ref{sec:enhanced-dissipation} contain
the proofs of Theorems~\ref{thm:main-thm-inviscid}
and~\ref{thm:main-thm-enhanced} and are designed in a similar manner: in both
cases we will first prove the lower bound, then introduce the concept of
$\rho$-irregularity (resp. Wei's condition) and explain its connection to the
upper bound, as well as to the irregularity of $u$; finally, we show by
probabilistic means that a.e. $u \in B^{\alpha}_{1, \infty}$ satisfies such
property. The end of Section~\ref{sec:enhanced-dissipation} also contains the
proof of Theorem~\ref{thm:main-thm-1}.

In Appendix~\ref{app:besov} we collect some well known results on Besov
spaces, while Appendix~\ref{app:wei-extension} contains a technical extension
of the results from~{\cite{wei}} needed to work in our setting.

\

\tmtextbf{Acknowledgments.} The authors were supported by the Deutsche
Forschungsgemeinschaft (DFG, German Research Foundation) through the Hausdorff
Center for Mathematics under Germany's Excellence Strategy -- EXC-2047/1 --
390685813 and through CRC 1060 - projekt number 211504053.

\

{\tmstrong{Notations and conventions.}} We will use the notation $a \lesssim
b$ to mean that there exists a constant $c > 0$ such that $a \leqslant c b$;
$a \lesssim_x b$ highlights the dependence $c = c (x)$. The notation $a \sim
b$ stands for $a \lesssim b$ and $b \lesssim a$, similarly for $a \sim_x b$.

Whenever needed, we will identify the $d$-dimensional torus $\mathbb{T}^d$
with either $[0, 2 \pi]^d$ or $[- \pi, \pi]^d$ with periodic boundary
condition, and functions $\varphi : \mathbb{T}^d \rightarrow \mathbb{R}$ with
$2 \pi$-periodic functions defined on $\mathbb{R}^d$. We will use
$d_{\mathbb{T}^d} (x, y)$ to denote the canonical distance on the flat torus
$\mathbb{T}^d$, namely $d_{\mathbb{T}^d} (x, y) = \inf_{k \in \mathbb{Z}^d} |
x + 2 \pi k - y |$, where $| \cdot |$ denotes the Euclidean distance on
$\mathbb{R}^d$. With a slight abuse, we will keep writing $| x |$ for $x \in
\mathbb{T}^d$ to denote $d_{\mathbb{T}^d} (x, 0)$.

$L^p (\mathbb{T}^d)$ denotes classical Lebesgue spaces, $C^{\alpha}
(\mathbb{T}^d)$ H\"{o}lder spaces and $H^s (\mathbb{T}^d) = W^{s, 2}
(\mathbb{T}^d)$ fractional Sobolev spaces.
$B^{\alpha}_{p, q} (\mathbb{T}^d)$ denotes Besov spaces on $\mathbb{T}^d$; we refer to Appendix~\ref{app:besov} for a detailed discussion of their definition and main properties. Here let us shortly recall, that for $\alpha\in (0,1)$ and $p\in [1,\infty)$, $f \in B^\alpha_{p,q}(\mathbb{T}^d)$ if and only if $f\in L^p(\mathbb{T}^d)$ and it has finite Gagliardo-Niremberd type seminorm
\begin{equation}
\llbracket f\rrbracket_{B^\alpha_{p,\infty}(\mathbb{T}^d)}:= \sup_{x \neq y
     \in \mathbb{T}^d} \frac{\left\| f \left( \cdot \, + x \right) - f \left(
     \cdot \, + y \right) \right\|_{L^p}}{d_{\mathbb{T}^d} (x, y)^s};
\end{equation}
see equations \eqref{eq:besov-equiv-norm-1}-\eqref{eq:besov-equiv-norm-2} for more details.
Similarly, $B^{\alpha}_{p, q} (0, \pi)$ denotes
Besov spaces on $[0, \pi]$.

Given $p \in [1, \infty)$ and a compact interval $I \subset
\mathbb{R}$, we denote by $V^p = V^p (I)$ the Banach space of functions $f : I
\rightarrow \mathbb{R}$ of finite $p$-variation, with norm
\[ \| f \|_{V^p} = | f (0) | + \sup_{\pi \in \Pi (I)} \left( \sum_{[t_i, t_{i
   + 1}] \in \pi} | f (t_{i + 1}) - f (t_i) |^p \right)^{\frac{1}{p}} \]
where the supremum is taken over the set $\Pi (I)$ of all finite partition of
$I$, identified with sequences $\{ t_i \}_{i = 0}^n$ such that $\min I = t_0 <
t_1 < \cdots < t_n = \max I$. $V^p_c$ stands for the closed subspace of $V^p$
of continuous functions. $V^p (\mathbb{T})$ is defined by identifying
$\mathbb{T}$ with the interval $[- \pi, \pi]$.

Whenever a stochastic process $X = (X_t)_{t \geqslant 0}$ is considered, if
not specified we tacitly assume the existence of an abstract underlying
filtered probability space $(\Omega, \mathcal{F}, (\mathcal{F}_t)_{t \geqslant
0}, \mathbb{P})$, such that the $\sigma$-algebra $\mathcal{F}$ and the
filtration $(\mathcal{F}_t)_{t \geqslant 0}$ satisfy the usual assumptions and
$(X_t)_{t \geqslant 0}$ is adapted to $(\mathcal{F}_t)_{t \geqslant 0}$.
Whenever we say that $(\mathcal{F}_t)_{t \geqslant 0}$ is the natural
filtration generated by $X$, then it is tacitly implied that it is actually
its right continuous, normal augmentation wrt. $\mathbb{P}$. We denote by
$\mathbb{E}$ integration (equiv. expectation) wrt. the probability
$\mathbb{P}$.

\section{Preliminaries}\label{sec:preliminaries}

\subsection{Prevalence}\label{sec:prevalence}

The theory of prevalence has been developed by Hunt, Sauer and Yorke
in~{\cite{huntsauer}} in order to provide a measure theoretic notion of
genericity in infinite dimensional spaces. It is a natural generalization of
the concept of ``full Lebesgue measure sets'' from the finite dimensional
setting. We follow here the exposition given in~{\cite{huntsauer}}, although
for our purposes it will be enough to work with Banach spaces $E$.

\begin{definition}
  \label{def:prevalence}Let $E$ be a complete metric vector space. A Borel set
  $A \subset E$ is said to be {\tmdfn{shy}} if there exists a measure $\mu$
  such that:
  \begin{enumerateroman}
    \item There exists a compact set $K \subset E$ such that $0 < \mu (K) <
    \infty$.
    
    \item For every $v \in E$, $\mu (v + A) = 0$.
  \end{enumerateroman}
  In this case, the measure $\mu$ is said to be {\tmdfn{transverse}} to $A$.
  More generally, a subset of $E$ is shy if it is contained in a shy Borel
  set. The complement of a shy set is called a {\tmdfn{prevalent}} set.
\end{definition}

Sometimes it is said informally that the measure $\mu$ ``witnesses'' the
prevalence of $A^c$.

It follows immediately from Point~\tmtextit{i.} of
Definition~\ref{def:prevalence} that, if such a measure $\mu$
exists, then it can be assumed to be a compactly supported probability measure
on $E$. On the other hand, in order to exhibit the existence of $\mu$
satisfying Points.~\tmtextit{i.}-\tmtextit{ii.}, it suffices to find another
{\tmem{tight}} probability measure $\tilde{\mu}$ only satisfying
requirement~\tmtextit{ii.} If $E$ is separable, then any probability measure
on $E$ is tight and therefore Point~\tmtextit{i.} is automatically satisfied.

\

The following properties hold for prevalence (all proofs can be found
in~{\cite{huntsauer}}):
\begin{enumeratenumeric}
  \item If $E$ is finite dimensional, then a set $A$ is shy if and only if it
  has zero Lebesgue measure.
  
  \item If $A$ is shy, then so is $v + A$ for any $v \in E$.
  
  \item Prevalent sets are dense.
  
  \item If $\dim (E) = + \infty$, then compact subsets of $E$ are shy.
  
  \item Countable union of shy sets is shy; conversely, countable intersection
  of prevalent sets is prevalent. 
\end{enumeratenumeric}
From now, whenever we say that a statement holds for a.e. $v \in E$, we mean
that the set of elements of $E$ for which the statement holds is a prevalent
set. Property~1. states that this convention is consistent with the finite
dimensional case.

In the context of a function space $E$, it is natural to consider as
probability measure the law induced by an $E$\mbox{-}valued random variable.
Namely, given stochastic process $W$ defined on a probability space $(\Omega,
\mathcal{F}, \mathbb{P})$ taking values in a separable Banach space $E$, in
order to show that a property $\mathcal{P}$ holds for a.e. $f \in E$, it
suffices to show that
\begin{equation}
  \mathbb{P} \left( \text{$f + W$ satisfies property $\mathcal{P}$} \right) =
  1, \qquad \forall \, f \in E. \label{eq:key-prevalence}
\end{equation}
Clearly, we are assuming that the set $A = \left\{ w \in E : \text{$w$
satisfies property $\mathcal{P}$} \right\}$ is Borel measurable; if $E$ is not
separable, we need to additionally require that the law of $W$ is tight, so as
to satisfy Point~\tmtextit{i.} of Definition~\ref{def:prevalence}.

As a consequence of properties~4. and~5., the set of all possible realizations
of a probability measure $\mu$ on a separable infinite dimensional Banach
space is a shy set, as it is contained in a countable union of compact sets
(this is true more in general for any tight measure on a Banach space). This
fact highlights the difference between a statement of the form ``Property
$\mathcal{P}$ holds for a.e. $f$ (in the sense of prevalence)'' and ``Property
$\mathcal{P}$ holds for $\mu$-a.e. $f$''; indeed, the second statement doesn't
provide any information regarding whether the property might be prevalent or
not. Intuitively, the elements satisfying a prevalence statement are ``many
more'' than just the realizations of a given measure $\mu$.

\subsection{A useful class of Gaussian transverse measures}\label{sec:fbm}

From now on, given an interval $[0, T]$ and a probability measure $\mu$ on $C
([0, T])$, we will denote by $(X_t)_{t \in [0, T]}$ the associated canonical
process, which is given by $X_t (\omega) = \omega (t)$ for $\omega \in C ([0,
T])$, and by $\mathcal{F}_t = \sigma (\{ X_s, s \leqslant t \})$ the
associated natural filtration.

A key point of the present work is to verify that suitable properties
$\mathcal{P}$ are satisfied by a.e. $f \in E$ for suitable $E = B^{\alpha}_{1,
\infty}$. The discussion from Section~\ref{sec:prevalence}, in particular
equation~\eqref{eq:key-prevalence}, suggests to look for classes of processes
which are stable under deterministic additive perturbations and
in~{\cite{galeati2020prevalence}} we identified the \tmtextit{local
nondeterministic} (LND) Gaussian processes as a useful class in the study of
prevalence in function spaces. We recall in the next definition that a real
valued process $X$ is Gaussian if for any $n \in \mathbb{N}$ and $t_1, \ldots,
t_n \in [0, T]$, $(X_{t_1}, \ldots, X_{t_n})$ is a $\mathbb{R}^n$-valued
Gaussian variable.

\begin{definition}
  \label{def:LND}Given $\beta > 0$, a real valued Gaussian process $X$ is
  {\tmdfn{strongly locally nondeterministic}} with parameter $\beta$,
  $\beta$-SLND for short, if there exists a constant $C_X$ such that
  \begin{equation}
    \tmop{Var} (X_t | \mathcal{F}_s) \geqslant C_X | t - s |^{2 \beta}
    \label{eq:LND}
  \end{equation}
  uniformly over $s, t \in [0, T]$ with $s < t$.
\end{definition}

In~\eqref{eq:LND} above, $\tmop{Var} \left( \cdot \, | \mathcal{F}_s \right)$
denotes the conditional variance; equivalently, Definition~\ref{def:LND}
amount to the condition that, for any $s < t$, there is a decomposition $X_t =
X^{(1)}_{s, t} + X^{(2)}_{s, t}$ where $X^{(1)}_{s, t}$ is Gaussian and
adapted to $\mathcal{F}_s$ while $X^{(2)}_{s, t}$ is Gaussian, independent of
$\mathcal{F}_s$, with variance $\tmop{Var} (X^{(2)}_{s, t}) \geqslant C_X  | t
- s |^{2 \beta}$. The increments of the process $X$ are therefore
``intrinsically chaotic'' in a way that can be quantified precisely by the
parameter $\beta$. Let us shortly mention that Definition~\ref{def:LND} is not
the only notion of LND in the literature and there are several non-equivalent
ones; see~{\cite{xiao}} for a review.

The importance of the $\beta$-SLND property comes from the following
elementary fact, which can be readily checked from the definition (see also
Remark~26 from~{\cite{galeati2020prevalence}}); in the statement, $f : [0, T]
\rightarrow \mathbb{R}$ can be naturally unbounded.

\begin{lemma}
  \label{lem:LND-translation}Let $\{ X_t \}_{t \in [0, T]}$ be a $\beta$-SLND
  Gaussian process and $f : [0, T] \rightarrow \mathbb{R}$ be a measurable
  function; then $X + f$ is also a $\beta$-SLND Gaussian process.
\end{lemma}

Lemma~\ref{lem:LND-translation} will be our main leverage to establish
prevalence statements, as it reduces the difficulty to that of verifying that
any $\beta$-SLND Gaussian process satisfies $\mu$-a.s. the property
$\mathcal{P}$ of interest; this will indeed be the strategy implemented in
Sections~\ref{sec:prevalence-mixing} and~\ref{sec:prevalence-enhanced}
respectively.

\

In this sense, we could work with any possible Gaussian law $\mu$ whose
associated canonical process is $\beta$-SLND, without further specification.
To keep things less abstract, we will however use a well-known one-parameter
family from this class, which are the laws $\left\{ \mu^H, \, H \in (0, 1)
\right\}$ of fractional Brownian motion (fBm) of parameter $H \in (0, 1)$. The
material recalled next is mostly classical and can be found in the
monograph~{\cite{nualart}}.

The law of fBm of Hurst parameter $H \in (0, 1)$ is defined as the unique
Gaussian measure $\mu^H$ on $\Omega = C ([0, T])$ such that
\[ \int_{\Omega} X_t (\omega) \mu^H (\mathd \omega) = 0, \quad \int_{\Omega}
   X_t (\omega) X_s (\omega) \mu^H (\mathd \omega) = \frac{1}{2} (| t |^{2 H}
   + | s |^{2 H} - | t - s |^{2 H}) . \]
For $H = 1 / 2$, the law of fBm corresponds to the classical Wiener measure;
instead for $H \neq 1 / 2$, the associated canonical process $X$ is not a
semimartingale nor a Markov process.

The support of $\mu^H$ in terms of Besov spaces is well understood, with sharp
results going back to~{\cite{ciesielski1993}} (see also~{\cite{veraar2009}}
for a modern proof which extends to the vector valued case): it holds
\[ \mu^H (C^{H - \varepsilon}) = 1 \quad \forall \, \varepsilon > 0, \quad
   \mu^H (B^H_{p, \infty}) = 1 \quad \forall \, p \in [1, \infty), \]
while
\[ \mu^H (C^H) = 0, \quad \mu^H (B^H_{p, q}) = 0 \quad \forall \, p, q \in [1,
   \infty) . \]
In particular fBm trajectories are sharply not $H$-H\"{o}lder continuous, but
by Ascoli--Arzel{\`a} $\mu^H$ is a tight probability measure on $B^{H -
\varepsilon}_{p, \infty}$ for any $\varepsilon > 0$ and any $p \in [1,
\infty]$. As promised, this class of Gaussian measures does satisfy the LND
property.

\begin{lemma}
  \label{lem:fBm-LND}Let $X$ be the canonical process associated to $\mu^H$,
  $H \in (0, 1)$. Then $X$ is $H$-SLND; moreover, the Gaussian process $Y_t :
  = \int_0^t X_s \mathd s$ is $(1 + H)$-SLND.
\end{lemma}

\begin{proof}
  The first claim is classical and can be found in the review~{\cite{xiao}}
  and the references therein; alternative, a self-contained proof, based on
  the Mandelbrot--Van Ness representation of fBm, is given in Section~2.4
  from~{\cite{galeati2020prevalence}}; the same representation can be used to
  establish the second half of the claim involving the process $Y$, see
  Example~\text{iv.} from Section~4.2 in~{\cite{galeati2020prevalence}}.
\end{proof}

Among the reasons for using $\mu^H$, instead of just any Gaussian measure
satisfying a suitable LND condition, let us finally mention that this process
can be simulated numerically in a very efficient way.

\section{Inviscid mixing}\label{sec:inviscid-mixing}

This section contains the proof of Theorem~\ref{thm:main-thm-inviscid}, which
we split in several steps.

Recall the setting: in order to study the transport equation $\partial_t f + u
\partial_x f = 0$, we pass to Fourier modes $f^k_t (y) = (P_k f_t) (y)$,
solving $\partial_t f^k + i k u f^k = 0$; namely $f^k_t (y) = e^{- i k t u
(y)} f^k_0 (y)$.

It is then natural to take a slightly more general perspective and study maps
of the form $y \mapsto e^{i \xi u (y)} g (y)$ with $\xi \in \mathbb{R}$, $g
\in H^s (\mathbb{T})$.

\subsection{Lower bounds in terms of
regularity}\label{sec:lower-bound-inviscid}

We show here that the regularity of $u$, measured in the Besov--Nikolskii
scale $B^{\alpha}_{1, \infty}$, necessarily implies a lower bound on the decay
of solutions in the $H^{- 1 / 2}$-norm. The proof is partly inspired by that
of Proposition~3.2 from~{\cite{colombozelati}}.

\begin{lemma}
  \label{lem:inviscid-lower-bound}Let $u \in B^{\alpha}_{1, \infty}
  (\mathbb{T})$ for some $\alpha \in (0, 1)$. Then for any $g \in H^1
  (\mathbb{T})$ there exists a constant $C = C (\alpha, g)$ such that
  \begin{equation}
    \| e^{i \xi u} g \|_{H^{- 1 / 2}} \geqslant C (1 + \| u \|_{B^{\alpha}_{1,
    \infty}})^{- \frac{1}{2 \alpha}}  | \xi |^{- \frac{1}{2 \alpha}} \quad
    \forall \, | \xi | \geqslant 1. \label{eq:inviscid-lower-bound}
  \end{equation}
\end{lemma}

\begin{proof}
  Fix $\xi$ with $| \xi | \geqslant 1$ and set $\bar{g} \assign e^{i \xi u}
  g$; we claim that $\bar{g} \in B^{\alpha / 2}_{2, \infty}$. By Sobolev and
  Besov embeddings, $g \in L^{\infty} \cap B^{\alpha / 2}_{2, \infty}$; $e^{i
  \xi u} \in L^{\infty}$, so it's enough to show that $e^{i \xi u} \in
  B^{\alpha / 2}_{2, \infty}$. By the basic estimate $| e^{i a} - e^{i b} |
  \leqslant \sqrt{2}  | a - b |^{1 / 2}$, it holds
  
  \begin{align*}
    \left\| e^{i \xi u \left( \cdot \, + y \right)} - e^{i \xi u (\cdot +
    \tilde{y})} \right\|_{L^2} & \lesssim | \xi |^{1 / 2} \left\| u \left(
    \cdot \, + y \right) - u \left( \cdot \, + \tilde{y} \right)
    \right\|_{L^1}^{1 / 2}\\
    & \lesssim | \xi |^{1 / 2} \| u \|^{1 / 2}_{B^{\alpha}_{1, \infty}}
    d_{\mathbb{T}} (y, \tilde{y})^{\alpha / 2} .
  \end{align*}
  By the equivalent characterization of Besov--Nikolskii spaces, this implies
  \[ \| e^{i \xi u} \|_{B^{\alpha / 2}_{2, \infty}} \lesssim 1 + | \xi |^{1 /
     2} \| u \|^{1 / 2}_{B^{\alpha}_{1, \infty}} \lesssim (1 + \| u
     \|_{B^{\alpha}_{1, \infty}})^{1 / 2} | \xi |^{1 / 2} \]
  and so by Proposition~\ref{prop:besov-algebra} in Appendix~\ref{app:besov}
  we conclude that $\bar{g} \in B^{\alpha / 2}_{2, \infty}$ with
  \begin{equation}
    \| \bar{g} \|_{B^{\alpha / 2}_{2, \infty}} \lesssim \| g \|_{H^1} (1 + \|
    u \|_{B^{\alpha}_{1, \infty}})^{1 / 2} | \xi |^{1 / 2} .
    \label{eq:inviscid-lower-proof1}
  \end{equation}
  Clearly $\| \bar{g} \|_{L^2} = \| g \|_{L^2}$. Using the interpolation
  inequality from Corollary~\ref{cor:interpolation-besov} in
  Appendix~\ref{app:besov} (for the choice $s_1 = 1 / 2$, $s_2 = \alpha / 2$)
  we obtain
  \begin{equation}
    \| g \|_{L^2} = \| \bar{g} \|_{L^2} \lesssim \| \bar{g} \|_{H^{- 1 /
    2}}^{\frac{\alpha}{1 + \alpha}}  \| \bar{g} \|_{B^{\alpha / 2}_{2,
    \infty}}^{\frac{1}{1 + \alpha}} . \label{eq:inviscid-lower-proof2}
  \end{equation}
  Rearranging now the terms in~\eqref{eq:inviscid-lower-proof2} and applying
  the estimate~\eqref{eq:inviscid-lower-proof1} we find
  \begin{equation} \label{eq:inviscid-lower-proof3}
      \| \bar{g} \|_{H^{- 1 / 2}}
       \gtrsim \| \bar{g} \|_{B^{\alpha /
      2}_{2, \infty}}^{- \frac{1}{\alpha}}  \| g \|_{L^2}^{1 +
      \frac{1}{\alpha}}
      \gtrsim \| g \|_{L^2}^{1 + \frac{1}{\alpha}} \| g \|_{H^1}^{-
      \frac{1}{\alpha}} (1 + \| u \|_{B^{\alpha}_{1, \infty}})^{- \frac{1}{2
      \alpha}} | \xi |^{- \frac{1}{2 \alpha}}
  \end{equation}
  where the hidden constant in~\eqref{eq:inviscid-lower-proof3} only depends
  on $\alpha$. Using the definition of $\bar{g}$ and relabelling the constant
  to include the $g$-dependent terms yields the conclusion.
\end{proof}

\begin{corollary}
  \label{cor:inviscid-lower-bound}Let $u \in B^{\alpha}_{1, \infty}
  (\mathbb{T})$ be mixing on $H^{1 / 2} (\mathbb{T})$ with rate
  $r_{1 / 2 \text{-mix}}$, in the sense of
  Definition~\ref{def:mixing}. Then there exists a constant $C = C (\alpha,
  u)$ such that
  \[ r_{1 / 2 \text{-mix}} (t) \leqslant C t^{\frac{1}{2
     \alpha}} . \]
\end{corollary}

\begin{proof}
  Consider $g (y) = e^{i y}$, so that $\| g \|_{H^{1 / 2}} \sim
  \| g \|_{H^1} \sim 1$; then by Definition~\ref{def:mixing} applied for the
  choice $k = 1$ and Lemma~\ref{lem:inviscid-lower-bound} for $\xi = - t$, it
  holds
  \begin{align*}
    \frac{1}{r (t)} \gtrsim & \| e^{- i t u} \|_{H^{1 / 2} \rightarrow H^{- 1
    / 2}} \geqslant \| e^{- i t u} g \|_{H^{- 1 / 2}} \gtrsim_{\alpha} (1 + \|
    u \|_{B^{\alpha}_{1, \infty}})^{- \frac{1}{2 \alpha}} t^{- \frac{1}{2
    \alpha}} ;
  \end{align*}
  up to relabelling constants, this yields the conclusion.
\end{proof}

\begin{remark}
  \label{rem:inviscid-lower-bound}In fact, the statement of
  Lemma~\ref{lem:inviscid-lower-bound} can be generalized as follows. For
  $\alpha \in (0, 1)$, $u \in B^{\alpha}_{1, \infty} (\mathbb{T})$, $g \in H^1
  (\mathbb{T})$ and any $s > 0$ there exists a constant $C (\alpha, g, s)$
  such that
  \[ \| e^{i \xi u} g \|_{H^{- s}} \geqslant C (1 + \| u \|_{B^{\alpha}_{1,
     \infty}})^{- \frac{s}{\alpha}}  | \xi |^{- \frac{s}{\alpha}} \quad
     \forall \, | \xi | \geqslant 1. \]
  Then arguing as in Corollary~\ref{cor:inviscid-lower-bound} by
  choosing $g (y) = e^{i y}$, one can conclude that the best possible rate
  for inviscid mixing on the scale $H^s (\mathbb{T})$ is $r_{s
  \text{-mix}} (t) \sim t^{s / \alpha}$. Taking $s = 1$ provides the rate
  $t^{1 / \alpha}$, which is in line with Proposition~3.2
  from~{\cite{colombozelati}}.
\end{remark}

\subsection{Upper bounds in terms of
$\rho$-irregularity}\label{sec:upper-bound-inviscid}

The concept of $\rho$-irregularity was first introduced
in~{\cite{catellier2016averaging}} in the study of regularization by noise
phenomena. Its applications to PDEs have been subsequently explored
in~{\cite{choukgubinelli1,choukgubinelli2,galeati2020prevalence,catellier2016rough}}.

\begin{definition}
  \label{def:rho-irr}Let $\gamma \in [0, 1)$, $\rho > 0$; a measurable map $u
  : [0, \pi] \rightarrow \mathbb{R}$ is said to be {\tmdfn{$(\gamma,
  \rho)$-irregular}} if there exists a constant $C > 0$ such that
  \begin{equation}
    \left| \int_I e^{i \xi u (z)} \mathd z \right| \leqslant C | I |^{\gamma} 
    | \xi |^{- \rho} \quad \forall \, \xi \in \mathbb{R}, \, I \subset [0,
    \pi] \label{eq:def-rho-irr}
  \end{equation}
  where $I$ stands for a subinterval of $[0, \pi]$ and $| I |$ denotes its
  length. A similar definition holds for $u : \mathbb{R} \rightarrow
  \mathbb{R}$; a map $u : \mathbb{T} \rightarrow \mathbb{R}$ is said to be
  $(\gamma, \rho)$-irregular if its $2 \pi$-periodic extension
  $u : \mathbb{R} \rightarrow \mathbb{R}$ has this property. We say that $u$
  is {\tmdfn{$\rho$-irregular}} for short if there exists $\gamma > 1 / 2$
  such that it is $(\gamma, \rho)$-irregular.
\end{definition}

In all of the cases covered by Definition~\ref{def:rho-irr},
following the original definition from~{\cite{catellier2016averaging}}, we
denote the optimal constant $C$ in~\eqref{eq:def-rho-irr} by $\| \Phi^u
\|_{\gamma, \rho}$. This is due to the notation $\Phi^u_t (\xi) : = \int_0^t
e^{i \xi u (z)} \mathd z$ and the fact that, for $u : [0, \pi] \rightarrow
\mathbb{R}$, by~\eqref{eq:def-rho-irr} it holds
\[ \| \Phi^u \|_{\gamma, \rho} = \sup_{\xi \in \mathbb{R}, 0 \leqslant s < t
   \leqslant \pi} \frac{| \Phi^u_t (\xi) - \Phi^u_s (\xi) |}{| t - s
   |^{\gamma} | \xi |^{- \rho}} . \]

The property of $\rho$-irregularity may be rephrased in the following form,
more suited for our purposes.

\begin{lemma}
  \label{lem:alternative-rho}Let $u : \mathbb{T} \rightarrow \mathbb{R}$ be
  $(\gamma, \rho)$-irregular, then
  \[ \| e^{i \xi u} \|_{B^{\gamma - 1}_{\infty, \infty}} \lesssim \| \Phi^u
     \|_{\gamma, \rho}  | \xi |^{- \rho} \quad \forall \, \xi \in \mathbb{R}.
  \]
\end{lemma}

\begin{proof}
  For $\bar{y} \in [- \pi, \pi]$ and $\xi \in \mathbb{R}$, define the function
  \begin{align*}
    v^{\xi} (\bar{y}) & = \int_{- \pi}^{\bar{y}} e^{i \xi u (y)} \mathd y -
    \left( \frac{\bar{y} + \pi}{2 \pi} \right) \int_{- \pi}^{\pi} e^{i \xi u
    (y)} \mathd y ;
  \end{align*}
  by periodicity it can be identified with a function on $\mathbb{T}$. Then by
  definition of $(\gamma, \rho)$-irregularity it holds $\| v^{\xi}
  \|_{C^{\gamma}} \lesssim \| \Phi^u \|_{\gamma, \rho}  | \xi |^{- \rho}$ and
  so by Proposition~\ref{prop:besov-differentiation} we deduce that
\begin{equation*}\begin{split}
	\| e^{i \xi u} \|_{B^{\gamma - 1}_{\infty \infty}}
	& = \left\|
       (v^{\xi})' + \frac{1}{2 \pi} \int_{- \pi}^{\pi} e^{i \xi u (y)} \mathd
       y \right\|_{B^{\gamma - 1}_{\infty, \infty}}\\
       & \lesssim \| v^{\xi} \|_{C^{\gamma}} + \frac{1}{2 \pi} \left|
       \int_{- \pi}^{\pi} e^{i \xi u (y)} \mathd y \right|\\
       & \lesssim \| \Phi^u \|_{\gamma, \rho}  | \xi |^{- \rho} .\qedhere
\end{split}\end{equation*}
\end{proof}

The relation between $\rho$-irregularity and inviscid mixing comes from the
next result.

\begin{lemma}
  \label{lem:estim-mixing-rho}Let $u : \mathbb{T} \rightarrow \mathbb{R}$ be
  $(\gamma, \rho)$-irregular for some $\gamma > 1 / 2$. Then there exists a
  constant $C = C (\gamma)$ such that
  \begin{equation}
    \| e^{i \xi u} g \|_{H^{- 1 / 2}} \leqslant C \| \Phi^u \|_{\gamma, \rho} 
    | \xi |^{- \rho} \| g \|_{H^{1 / 2}}  \quad \forall \, \xi \neq 0, \, g
    \in H^{1 / 2} . \label{eq:estim-mixing-rho}
  \end{equation}
  As a consequence, $u$ is mixing on the scale $H^{1 / 2}$ with rate
  $r_{1 / 2 \text{-mix}} (t) = t^{\rho}$, in the sense of
  Definition~\ref{def:mixing}.
\end{lemma}

\begin{proof}
  The proof of the estimate~\eqref{eq:estim-mixing-rho} relies on several
  properties of Besov spaces, for which we refer the reader to
  Appendix~\ref{app:besov}. By assumption $\gamma + 1 / 2 > 1$,
  thus we can apply Proposition~\ref{prop:besov-paraproducts}
  (for the choice $s_1 = \gamma - 1$, $s_2 = 1 / 2$, $p_1 = q = \infty$, $p_2
  = p = 2$) and Lemma~\ref{lem:alternative-rho} to obtain
  \begin{align*}
    \| e^{i \xi u} g \|_{B^{\gamma - 1}_{2, \infty}} & \lesssim \| e^{i \xi u}
    \|_{B^{\gamma - 1}_{\infty, \infty}} \| g \|_{B^{1 / 2}_{2, \infty}}\\
    & \lesssim \| \Phi^u \|_{\gamma, \rho}  | \xi |^{- \rho} \| g \|_{B^{1 /
    2}_{2, 2}}\\
    & = \| \Phi^u \|_{\gamma, \rho}  | \xi |^{- \rho} \| g \|_{H^{1 / 2}} .
  \end{align*}
  Again by the hypothesis $\gamma - 1 > - 1 / 2$ and so by Besov embeddings
  $B^{\gamma - 1}_{2, \infty} \hookrightarrow H^{- 1 / 2}$, yielding the first
  claim. Applying estimate~\eqref{eq:estim-mixing-rho} for $k \in
  \mathbb{Z}_0$, $\xi = - t k$ gives
  \[ \| e^{- i t k u} \|_{H^{1 / 2} \rightarrow H^{- 1 / 2}} \leqslant \frac{C
     \| \Phi^u \|_{\gamma, \rho} }{(t | k |)^{\rho}} \]
  and thus the conclusion.
\end{proof}

The property of $\rho$-irregularity implies roughness of $u$, as the name
suggests. To quantify this precisely, we recall the concept of H\"{o}lder
roughness, as presented in~{\cite{frizhairer}}.

\begin{definition}
  \label{def:holder-roughness}A measurable map $u : \mathbb{T} \rightarrow
  \mathbb{R}$ is said to be $\alpha$-H\"{o}lder {\tmdfn{rough}} if there
  exists $L = L_{\alpha} (u)$ such that: for any $\delta > 0$ and any $\bar{y}
  \in \mathbb{T}$, there exists $z \in \mathbb{T}$ satisfying
  \[ d_{\mathbb{T}} (\bar{y}, z) \leqslant \delta \quad \text{ and } \quad | u
     (\bar{y}) - u (z) | \geqslant L_{\alpha} (u) \delta^{\alpha} . \]
  The optimal constant $L_{\alpha} (u)$ is called the modulus of
  $\alpha$-H\"{o}lder roughness of $u$.
\end{definition}

Definition~\ref{def:holder-roughness} is equivalent to requiring
\begin{equation}
  L_{\alpha} (u) = \inf_{\bar{y} \in \mathbb{T}, \delta > 0} \sup_{z \in
  B_{\delta} (\bar{y})} \frac{| u (z) - u (\bar{y}) |}{\delta^{\alpha}} > 0.
\end{equation}
A detailed study of analytic properties of $\rho$-irregular paths was carried
out in Section~5 of~{\cite{galeati2020prevalence}}; in
particular, there exists a critical prameter $\alpha^{\ast}$, associated to
the pair $(\gamma, \rho)$, linked to the (ir)regularity of $u$ in H\"{o}lder
and Besov--Nikolskii scales.

\begin{proposition}
  \label{prop:rho-irr-besov-nikolskii}Let $u : \mathbb{T} \rightarrow
  \mathbb{R}$ be $(\gamma, \rho)$-irregular and define $\alpha^{\ast} \assign
  (1 - \gamma) / \rho$. Then:
  \begin{enumeratealpha}
    \item $u$ is $\alpha$-H\"{o}lder rough for any $\alpha > \alpha^{\ast}$
    with $L_{\alpha} (u) = + \infty$.
    
    \item $u$ has infinite $p$-variation on any subinterval $I \subset
    \mathbb{T}$ and for any $p > 1 / \alpha^{\ast}$.
    
    \item $u$ does not belong to $B^{\alpha}_{1, \infty}$ for
    any $\alpha > \alpha^{\ast}$.
  \end{enumeratealpha}
\end{proposition}

\begin{proof}
  For functions $u : [0, T] \rightarrow \mathbb{R}$,
  points~{\tmem{a}}) and~{\tmem{b}}) are proved
  in~{\cite{galeati2020prevalence}}, cf. Corollary~65 and Corollary~68
  therein; we recall here shortly the idea of proof.
  
  Going through the proof of Theorem~63 from~{\cite{galeati2020prevalence}},
  one can establish the (much stronger) fact that, if $u$ is $(\gamma,
  \rho)$-irregular, then for any $\tilde{\alpha} > \alpha^{\ast}$ it holds
  \begin{equation}
    \lim_{\varepsilon \rightarrow 0^+} \inf_{y \in (0, T)} \varepsilon^{- 1}
    \mathcal{L} (h \in (0, \varepsilon) : | u (y + h) - u (y) | \geqslant
    \varepsilon^{\tilde{\alpha}}) = 1, \label{eq:key-irregularity}
  \end{equation}
  where $\mathcal{L}$ denotes the Lebesgue measure on $\mathbb{R}$. In
  particular, there exists $\varepsilon_0 > 0$ such that, for all $0 <
  \varepsilon < \varepsilon_0$, it must hold
  \[ \mathcal{L} (h \in (0, \varepsilon) : | u (y + h) - u (y) | \geqslant
     \varepsilon^{\tilde{\alpha}}) \geqslant \varepsilon / 2 > 0 \quad \forall
     \, y \in (0, T) ; \]
  therefore for any $y \in (0, T)$ we can find infinitely many, arbitrarily
  small $h$ such that $| u (y + h) - u (y) | \geqslant h^{\tilde{\alpha}}$;
  playing with the arbitrariness of $\tilde{\alpha}$, one can then easily
  establish both properties of H\"{o}lder roughness and infinite
  $p$\mbox{-}variation.
  
  Up to identifying $u : \mathbb{T} \rightarrow \mathbb{R}$ with a $2
  \pi$-periodic function, it's easy to check that
  property~\eqref{eq:key-irregularity} carries over to this setting as well,
  as it is only related to the local behaviour or $u$ around any fixed $y$;
  same goes for the proofs of points ~{\tmem{a}}) and~{\tmem{b}}).
  
  We now focus on establishing claim~\tmtextit{c}), which is instead an
  original contribution of this work. Fix $\alpha > \alpha^{\ast}$ and choose
  $\tilde{\alpha} \in (\alpha^{\ast}, \alpha)$; by
  estimate~\eqref{eq:key-irregularity} (with the infimum taken over $y \in
  \mathbb{T}$ instead of $(0, T)$),
  for all $\varepsilon > 0$ sufficiently small, it must hold
  \begin{align*}
    \pi & \leqslant \int_{\mathbb{T}} \varepsilon^{- 1} \mathcal{L} (h \in (0,
    \varepsilon) : | u (y + h) - u (y) | \geqslant
    \varepsilon^{\tilde{\alpha}}) \mathd y\\
    & \leqslant \int_{\mathbb{T}} \varepsilon^{- 1 - \tilde{\alpha}}
    \int_0^{\varepsilon} | u (y + h) - u (y) | \mathd h \mathd y\\
    & = \varepsilon^{- 1 - \tilde{\alpha}} \int_0^{\varepsilon} \left\| u
    \left( \cdot \, + h \right) - u (\cdot) \right\|_{L^1} \mathd h\\
    & \leqslant \varepsilon^{\alpha - \tilde{\alpha}}  \llbracket u
    \rrbracket_{B^{\alpha}_{1, \infty}},
  \end{align*}
  where in the second passage we used Markov's inequality. Since $\alpha >
  \tilde{\alpha}$, letting $\varepsilon \rightarrow 0^+$ we can conclude that
  $\llbracket u \rrbracket_{B^{\alpha}_{1, \infty}} = + \infty$.
\end{proof}

\begin{remark}
  If $u$ is $\rho$-irregular, then
  Proposition~\ref{prop:rho-irr-besov-nikolskii}\mbox{-}c)
  implies that $u$ does not belong to $B^{\alpha}_{1, \infty}$ for any $\alpha
  > (2 \rho)^{- 1}$. Conversely, if $u \in B^{\alpha}_{1, \infty}$, then it
  can only be $\rho$-irregular for parameters $\rho$ satisfying $\rho
  \leqslant (2 \alpha)^{- 1}$.
\end{remark}

\subsection{Prevalence statements and proof of
Theorem~\ref{thm:main-thm-inviscid}}\label{sec:prevalence-mixing}

Given the results of
Sections~\ref{sec:lower-bound-inviscid}--\ref{sec:upper-bound-inviscid}, it is
natural to wonder whether generic elements of $B^{\alpha}_{1, \infty}$ are
``almost as irregular as possible'', in the sense of being $\rho$-irregular
for any $\rho < (2 \alpha)^{- 1}$; we provide here a positive answer.

In order to do so, we will first prove the statement for elements of
$B^{\alpha}_{1, \infty} (0, \pi)$, see
Theorem~\ref{thm:prevalence-interval-inviscid}, and only later deduce the same
property for $B^{\alpha}_{1, \infty} (\mathbb{T})$ by a ``deperiodization''
procedure (cf. Corollary~\ref{cor:prevalence-periodic-inviscid} below).

Differently from Section~\ref{sec:fbm}, whenever dealing with a measure $\mu$
supported on $C ([0, \pi])$, it will be useful to denote by $u = \{ u_y \}_{y
\in [0, \pi]}$ the associated canonical process; we will instead employ the
letter $\varphi$ to denote deterministic functions, either defined on $[0,
\pi]$ or on $\mathbb{T}$.

Before proceeding further, we need to recall the following key result
established in~{\cite{galeati2020prevalence}}, cf. Theorem~29 therein.

\begin{proposition}
  \label{thm:prevalence-rho-irregularity}Let $\mu$ be a Gaussian measure on $C
  ([0, T])$ whose canonical process $u$ is $\beta$-SLND for some $\beta > 0$.
  Then for any $\rho < (2 \beta)^{- 1}$ it holds
  \[ \mu^H \left( u \text{ is } \rho \text{-irregular} \right) = 1. \]
\end{proposition}

We can combine Proposition~\ref{thm:prevalence-rho-irregularity} with the
invariance of the $\beta$-SLND property from Lemma~\ref{lem:LND-translation}
to deduce a first prevalence statement.

\begin{theorem}
  \label{thm:prevalence-interval-inviscid}Let $\alpha \in (0, 1)$; then a.e.
  $\varphi \in B^{\alpha}_{1, \infty} (0, \pi)$ is $\rho$-irregular for every
  $\rho < (2 \alpha)^{- 1}$.
\end{theorem}

\begin{proof}
  Given $\rho > 0$, define the set
  \[ \mathcal{A}_{\rho} = \left\{ \varphi \in B^{\alpha}_{1, \infty} (0, \pi)
     : \varphi \text{ is $\rho$-irregular} \right\} ; \]
  it holds
  \[ \mathcal{A}_{\rho} = \bigcup_{n, m = 3}^{\infty} \mathcal{A}_{\rho, n,
     m}, \]
  with
  \[ \mathcal{A}_{\rho, n, m} \assign \left\{ \varphi \in B^{\alpha}_{1,
     \infty} (0, \pi) : \varphi \text{ is $(\gamma, \rho)$-irr. for } \gamma =
     \frac{1}{2} + \frac{1}{n}, \| \Phi^{\varphi} \|_{\gamma, \rho} \leqslant
     m  \right\} . \]
  The sets $\mathcal{A}_{\rho, n, m}$ are closed in the topology of
  $B^{\alpha}_{1, \infty} (0, \pi)$ (the map $\varphi \mapsto \|
  \Phi^{\varphi} \|_{\gamma, \rho}$ is lower semicontinuous in the topology of
  $L^1 (0, \pi)$), thus $\mathcal{A}_{\rho}$ is Borel measurable. If we show
  that $\mathcal{A}_{\rho}$ is prevalent in $B^{\alpha}_{1, \infty} (0, \pi)$
  for any $\rho < (2 \alpha)^{- 1}$, then the same holds for
  \[ \mathcal{A} = \left\{ \varphi \in B^{\alpha}_{1, \infty} (0, \pi) :
     \varphi \text{ is $\rho$-irregular for every } \rho < \frac{1}{2 \alpha}
     \right\} = \bigcap_{n = 1}^{\infty} \mathcal{A}_{\frac{1}{2 \alpha} -
     \frac{1}{n}} \]
  providing the conclusion.
  
  Now fix $\rho < (2 \alpha)^{- 1}$ and choose $H \in (0, 1)$ such that $H >
  \alpha$, $\rho < (2 H)^{- 1}$; denote by $\mu^H$ the law of fractional
  Brownian motion on $C ([0, \pi])$ and by $u = \{ u_y, y \in [0, \pi] \}$ the
  associated canonical process. Since $\mu^H$ is supported on $C^{H -
  \varepsilon} ([0, \pi])$ for any $\varepsilon > 0$ and $H > \alpha$, it is
  also a tight probability measure on $B^{\alpha}_{1, \infty} (0, \pi)$; thus
  we only need to verify Property~\tmtextit{ii.} from
  Definition~\ref{def:prevalence}, equivalently
  property~\eqref{eq:key-prevalence} for $E = B^{\alpha}_{1, \infty}$.
  
  Fix $\varphi \in B^{\alpha}_{1, \infty} (0, \pi)$; by
  Proposition~\ref{lem:fBm-LND}, $u$ is a $H$-SLND process and so by
  Lemma~\ref{lem:LND-translation} the same holds for $u + \varphi$. In turn,
  by our choise of the parameters and
  Proposition~\ref{thm:prevalence-rho-irregularity}, this implies that
  $\varphi + u$ is $\mu^H$-a.s. $\rho$-irregular; as the argument holds for
  any $\varphi \in B^{\alpha}_{1, \infty} (0, \pi)$, we have shown that
  \[ \mu^H (\varphi + \mathcal{A}_{\rho}) = 1 \quad \forall \, \varphi \in
     B^{\alpha}_{1, \infty} (0, \pi), \]
  namely that $\mu^H$ witnesses the prevalence of $\mathcal{A}_{\rho}$ in
  $B^{\alpha}_{1, \infty} (0, \pi)$.
\end{proof}

We pass to show how to exploit Theorem~\ref{thm:prevalence-interval-inviscid}
to establish similar statement for functions defined on the torus.

We identify the torus $\mathbb{T}$ with the interval $[- \pi, \pi]$, up to $-
\pi \sim \pi$; thus any measurable function $\varphi : \mathbb{T} \rightarrow
\mathbb{R}$ can be identified with $\varphi : [- \pi, \pi] \rightarrow
\mathbb{R}$ such that $\varphi (- \pi) = \varphi (\pi)$. Any such $\varphi$ is
in a 1-1 correspondence with a pair $(\varphi_1, \varphi_2)$ of measurable
functions defined on $[0, \pi]$, given by $\varphi_1 (y) : = \varphi (y)$,
$\varphi_2 (y) : = \varphi (- y)$; they satisfy the constraint $\varphi_1
(\pi) = \varphi_2 (\pi)$. The $\rho$-irregularity property of the periodic
function $\varphi$ is actually equivalent to that of the aperiodic functions
$\varphi_i$.

\begin{lemma}
  \label{lem:techlem-rho}A measurable function $\varphi : \mathbb{T}
  \rightarrow \mathbb{R}$ is $(\gamma, \rho)$-irregular if and only if the
  functions $\varphi_1, \varphi_2 : [0, \pi] \rightarrow \mathbb{R}$ are so.
\end{lemma}

\begin{proof}
  The proof is elementary. Given $I \subset [- \pi, \pi]$, setting $I_1 = I
  \cap [0, \pi]$, $I_2 = I \cap [- \pi, 0]$ it holds $\max \{ | I_1 |, | I_2 |
  \} \leqslant | I | \leqslant 2 \max \{ | I_1 |, | I_2 | \}$, so that
  \[ \max \{ \| \Phi^{\varphi_1} \|_{\gamma, \rho}, \| \Phi^{\varphi_2}
     \|_{\gamma, \rho} \} \leqslant \| \Phi^{\varphi} \|_{\gamma, \rho}
     \leqslant 2 \max \{ \| \Phi^{\varphi_1} \|_{\gamma, \rho}, \|
     \Phi^{\varphi_2} \|_{\gamma, \rho} \} . \qedhere \]
\end{proof}

Conversely, given a measurable $\tilde{\varphi} : [0, \pi] \rightarrow
\mathbb{R}$, we can associate it another function $\varphi = T \tilde{\varphi}
: \mathbb{T} \rightarrow \mathbb{R}$ by setting $T \tilde{\varphi} (y) =
\tilde{\varphi} (| y |)$, which corresponds to $(T \tilde{\varphi})_1 = (T
\tilde{\varphi})_2 = \tilde{\varphi} $. It immediately follows from
Lemma~\ref{lem:techlem-rho} that $T \tilde{\varphi}$ is $(\gamma,
\rho)$-irregular if and only if $\tilde{\varphi}$ is so; it is also easy to
check that, if $\tilde{\varphi} \in B^{\alpha}_{1, \infty} (0, \pi) \cap
L^{\infty} (0, \pi)$, then $T \tilde{\varphi} \in B^{\alpha}_{1, \infty}
(\mathbb{T})$.

We are finally ready to prove a prevalence statement in $B^{\alpha}_{1,
\infty} (\mathbb{T})$.

\begin{corollary}
  \label{cor:prevalence-periodic-inviscid}Let $\alpha \in (0, 1)$, then a.e.
  $\varphi \in B^{\alpha}_{1, \infty} (\mathbb{T})$ is $\rho$-irregular for
  any $\rho < (2 \alpha)^{- 1}$.
\end{corollary}

\begin{proof}
  The proof that the set
  \[ \mathcal{A} \assign \left\{ \varphi \in B^{\alpha}_{1, \infty}
     (\mathbb{T}) : \varphi \text{ is } \rho \text{-irregular for any } \rho <
     \frac{1}{2 \alpha}  \right\} \]
  is Borel in the topology of $B^{\alpha}_{1, \infty} (\mathbb{T})$ is
  identical to that of Theorem~\ref{thm:prevalence-interval-inviscid} and thus
  omitted; as therein, we can introduce the sets $\mathcal{A}_{\rho}$ and
  reduce the task to establish the prevalence of the set $\mathcal{A}_{\rho}$
  for any fixed $\rho < (2 \alpha)^{- 1}$.
  
  Choose $H \in (0, 1)$ such that $H > \alpha$, $\rho < (2 H)^{- 1}$ and
  denote by $\mu^H$ the associated law of fBm; since it is supported on
  $B^{\alpha}_{1, \infty} (0, \pi) \cap L^{\infty} (0, \pi)$, we can define a
  new measure on $B^{\alpha}_{1, \infty} (\mathbb{T})$ by $\nu^H : =
  T_{\sharp} \mu^H$, where $(T \tilde{\varphi}) (y) = \tilde{\varphi} (| y |)$
  for $y \in [0, \pi]$ and $T_{\sharp}$ denotes the pushforward measure.
  
  Recall the notation $\varphi_1, \varphi_2$ from Lemma~\ref{lem:techlem-rho};
  for any $\varphi \in B^{\alpha}_{1, \infty} (\mathbb{T})$ it holds
  \begin{align*}
    \nu^H (\varphi + \mathcal{A}) & = \mu^H \left( \left\{ u \in
    B^{\alpha}_{1, \infty} (0, \pi) : \, T u + \varphi \text{ is } \rho
    \text{-irregular}  \right\} \right)\\
    & = \mu^H \left( \bigcap_{i = 1}^2 \left\{ u \in B^{\alpha}_{1, \infty}
    (0, \pi) : \, u + \varphi_i \text{ is } \rho \text{-irregular}  \right\}
    \right) = 1 ;
  \end{align*}
  in the last passage we used the already established properties of the
  measure $\mu^H$ from the proof of
  Theorem~\ref{thm:prevalence-interval-inviscid}, as well as the fact that the
  intersection of sets of full measure is still of full measure. Overall, this
  shows that $\nu^H$ witnesses the prevalence of the set $\mathcal{A}_{\rho}$;
  the conclusion follows using the fact that countable intersection of
  prevalent sets is prevalent.
\end{proof}

We are now ready to complete the

\begin{proof}[Proof of Theorem~\ref{thm:main-thm-inviscid}]
  The lower bound comes from Corollary~\ref{cor:inviscid-lower-bound}, while
  the upper bound from a combination of Lemma~\ref{lem:estim-mixing-rho} and
  Corollary~\ref{cor:prevalence-periodic-inviscid}.
\end{proof}

\section{Enhanced dissipation}\label{sec:enhanced-dissipation}

This section contains the proof of Theorem~\ref{thm:main-thm-enhanced} split
in several steps.

Recall the setting: we want to study the asymptotic behavior of the family of
complex-valued PDEs~\eqref{eq:harmonic-oscillator}, equivalently obtain upper
and lower bounds on
\[ \| e^{t L_{k, \nu}} \|_{L^2 (\mathbb{T}; \mathbb{C}) \rightarrow L^2
   (\mathbb{T}; \mathbb{C})} \quad \text{as } t \rightarrow \infty, \]
where $L_{k, \nu} \assign - i k u + \nu \partial_y^2 .$

\subsection{Lower bounds in terms of
regularity}\label{sec:lower-bound-enhanced}

We show here that if $u$ has regularity of degree $\alpha \in (0, 1)$, as
measured in a suitable Besov--Nikolskii scale, then the its best possible
diffusion enhancing rate is $r_{\text{dif}}(\nu) \sim
\nu^{\alpha / (2 + \alpha)}$. The precise statement goes as follows.

\begin{proposition}
  \label{prop:lower-bound-enhanced}Let $u \in B^{\alpha}_{1, \infty}
  (\mathbb{T})$ be diffusion enhancing with rate
  $r_{\text{dif}}$, in the sense of
  Definition~\ref{def:diffusion-enhancing}; then there exists a constant $C >
  0$ such that
  \[ r_{\text{dif}} (\nu) \leqslant C \nu^{\frac{\alpha}{\alpha
     + 2}} \]
  for all $\nu \in (0, 1]$.
\end{proposition}

In order to provide estimates for $e^{t L_{k, \nu}}$ it is convenient to study
more generally the properties of solutions $g : \mathbb{T} \rightarrow
\mathbb{C}$ to
\begin{equation}
  \partial_t g + i \xi u g = \nu \partial_y^2 g \label{eq:complex-PDE}
\end{equation}
in function of the parameters $\xi \in \mathbb{R}$, $\nu \in (0, 1)$ and the
shear flow $u$.

The proof of Proposition~\ref{prop:lower-bound-enhanced} follows a similar
strategy to~{\cite{zelatidrivas}} and is based on deriving a Lagrangian
Fluctuation-Dissipation relation (FDR) for the PDE~\eqref{eq:complex-PDE},
which is a result of independent interest.

\begin{proposition}
  \label{prop:FDR}Let $u \in L^1 (\mathbb{T})$, $g$ be a solution
  to~\eqref{eq:complex-PDE} with initial data $g_0 \in L^2 (\mathbb{T};
  \mathbb{C})$; for any $(t, y) \in \mathbb{R}_{\geqslant 0} \times
  \mathbb{T}$, define the complex random variable
  \[ Z^y_t = \exp \left( - i \xi \int_0^t u \left( y + \sqrt{2 \nu} B_s
     \right) \mathd s \right) g_0 \left( y + \sqrt{2 \nu} B_t \right) \]
  where $B$ is a standard real-valued BM. Then we have the following
  Lagrangian FDR:
  \begin{equation}
    \| g_0 \|_{L^2}^2 - \| g_t \|_{L^2}^2 = \int_{\mathbb{T}} \tmop{Var}
    (Z^y_t) \mathd y. \label{eq:FDR}
  \end{equation}
\end{proposition}

\begin{proof}
  Without loss of generality, we can assume $u$ and $g_0$ to be smooth, as
  identity~\eqref{eq:FDR} in the general case will follow from an approximation
  argument (the definition of $Z^y_t$ is meaningful for any $u \in L^1
  (\mathbb{T})$, thanks to the properties of the local time of a Brownian
  motion).
  Let us however first show that the r.h.s. of \eqref{eq:FDR} is a well-defined quantity, which can be estimated independently of the smoothness of $u$, $g_0$. Indeed, for any $t\geq 0$ it holds
\begin{align*}
\int_{\mathbb{T}} \mathbb{E}[|Z^y_t|^2] \mathd y
& = \int_{\mathbb{T}} \mathbb{E}\big[\,|g_0|^2(y+\sqrt{2\nu} B_t)\big] \mathd y
= \mathbb{E}\left[ \int_{\mathbb{T}} |g_0|^2(y+\sqrt{2\nu} B_t) \mathd y\right]
= \| g_0\|_{L^2}^2,
\end{align*}
where in the last step we used the invariance of the $L^2$-norm of $g_0$ under (random) translations; the pointwise bound $\tmop{Var} (Z^y_t)\leq \mathbb{E}[|Z^y_t|^2]$ then readily yields an estimate for the r.h.s. of \eqref{eq:FDR}.
  
  Now, by the Feynman--Kac formula, the solution $g$
  to~\eqref{eq:complex-PDE} is given by $g (t, y) =\mathbb{E} [Z^y_t]$.
  Moreover since $u$ is real valued, we have the energy balance
  \[ \| g_0 \|_{L^2}^2 - \| g_t \|_{L^2}^2 = 2 \nu \int_0^t \| \partial_y g_s
     \|_{L^2}^2 \mathd s ; \]
  and more generally, the map $(t, y) \mapsto | g |^2 (t,
  y)$ satisfies
  \[ \partial_t | g |^2 = \nu \partial_y^2 | g |^2 - 2 \nu | \partial_y g |^2
     . \]
  Now let $h$ to be a solution of $\partial_t h = \nu \partial_y^2 h$ with
  initial data $h_0 = | g_0 |^2$. It holds
  \[ \frac{\mathd}{\mathd t} \int_{\mathbb{T}} [| g |^2 - h] \mathd
     y = - 2 \nu \| \partial_y g \|_{L^2}^2, \]
  which implies that
  \[ \| g_0 \|_{L^2}^2 - \| g_t \|_{L^2}^2 = 2 \nu \int_0^t \| \partial_y g
     \|_{L^2}^2 = \int_{\mathbb{T}} [h_t (y) - | g_t (y) |^2] \mathd y. \]
  Finally, since by Feynman--Kac, $h (t, y) =\mathbb{E} \left[ | g_0|^2( y + \sqrt{2 \nu} B_t)  \right]$, we obtain
  \begin{align*}
       \| g_0 \|_{L^2}^2 - \| g_t \|_{L^2}^2
       & = \int_{\mathbb{T}} \left(
       \mathbb{E} \left[ | g_0 |^2( y + \sqrt{2 \nu} B_t)
       \right] - | \mathbb{E} [Z^y_t] |^2 \right) \mathd y\\
       & = \int_{\mathbb{T}} (\mathbb{E} [| Z^y_t |^2] - | \mathbb{E}
       [Z^y_t] |^2) \mathd y
  \end{align*}
  which gives the conclusion.
\end{proof}

\begin{lemma}
  \label{lem:lower-bound-enhanced-1}Let $g_0 \in H^1 (\mathbb{T};
  \mathbb{C})$, $u \in B^{\alpha}_{1, \infty} (\mathbb{T})$ for some $\alpha
  \in (0, 1)$ and $\xi \in \mathbb{R}$. Then there exists $C = C (\alpha) > 0$
  such that the solution $g$ to~\eqref{eq:complex-PDE} satisfies
  \[ \| g_0 \|_{L^2}^2 - \| g_t \|_{L^2}^2 \leqslant C \| g_0 \|^2_{H^1}
     \left( \nu t + \llbracket u \rrbracket_{B^{\alpha}_{1, \infty}} | \xi |
     \nu^{\frac{\alpha}{2}} t^{1 + \frac{\alpha}{2}} \right) \quad \forall \,
     t, \nu > 0. \]
\end{lemma}

\begin{proof}
  Recall the elementary identity $2 \tmop{Var} (X) =\mathbb{E} [| X -
  \tilde{X} |^2]$ for $\tilde{X}$ being an i.i.d. copy of $X$. In our setting,
  we take
  \[ \tilde{Z}^y_t = \exp \left( - i \xi \int_0^t u \left( y + \sqrt{\nu}
     \tilde{B}_s \right) \mathd s \right) g_0 \left( y + \sqrt{\nu}
     \tilde{B}_t \right) \]
  where $\tilde{B}$ is another BM independent of $B$. Therefore
  \begin{align*}
    \| g_0 \|_{L^2}^2 - \| g_t \|_{L^2}^2 & = \, \frac{1}{2} 
    \int_{\mathbb{T}} \mathbb{E} [| Z^y_t - \tilde{Z}^y_t |^2] \mathd y\\
    & \leqslant \, \mathbb{E} \left[ \int_{\mathbb{T}} \left| g_0 \left( y +
    \sqrt{\nu} B_t \right) - g_0 \left( y + \sqrt{\nu} \tilde{B}_t \right)
    \right|^2 \mathd y \right]\\
    & + \| g_0 \|^2_{L^{\infty}} \mathbb{E} \left[ \int_{\mathbb{T}} \left|
    e^{- i \xi \int_0^t u \left( y + \sqrt{\nu} B_s \right) \mathd s} - e^{- i
    \xi \int_0^t u \left( y + \sqrt{\nu} \tilde{B}_s \right) \mathd s}
    \right|^2 \mathd y \right] .
  \end{align*}
  Using the inequality $| e^{i \xi a} - e^{i \xi b} | \leqslant \sqrt{2}  |
  \xi |^{1 / 2} | b - a |^{1 / 2}$ and the characterization of Besov spaces in
  terms of finite differences (see Appendix~\ref{app:besov}), we deduce
  \begin{align*}
    \| g_0 \|_{L^2}^2 - \| g_t \|_{L^2}^2 & \lesssim \, \mathbb{E} \left[
    \left\| g_0 \left( \cdot \, + \sqrt{\nu} B_t \right) - g_0 \left( \cdot \,
    + \sqrt{\nu} B_t \right) \right\|_{L^2}^2 \right] \\
    & \quad + \| g_0 \|_{L^{\infty}}^2  | \xi | \mathbb{E} \left[
    \int_{\mathbb{T}} \int_0^t \left| u \left( y + \sqrt{\nu} B_s \right) - u
    \left( y + \sqrt{\nu} \tilde{B}_s \right) \right| \mathd s \mathd y
    \right]\\
    & \lesssim \| g_0 \|^2_{H^1} \left( \nu \mathbb{E} [| B_t - \tilde{B}_t
    |^2] + | \xi | \int_0^t \mathbb{E} \left[ \left\| u \left( \cdot \, +
    \sqrt{\nu} B_s \right) - u \left( \cdot \, + \sqrt{\nu} \tilde{B}_s
    \right) \right\|_{L^1} \right] \mathd s \right)\\
    & \lesssim \| g_0 \|^2_{H^1} \left( \nu t + \llbracket u
    \rrbracket_{B^{\alpha}_{1, \infty}} | \xi | \nu^{\frac{\alpha}{2}}
    \int_0^t \mathbb{E} [| B_s - \tilde{B}_s |^{\alpha}] \mathd s \right) ;
  \end{align*}
  computing the last expectation yields the conclusion.
\end{proof}

We are now ready to complete the

\begin{proof}[Proof of Proposition~\ref{prop:lower-bound-enhanced}]
  The proof goes along the same lines as Lemma~2 from~{\cite{zelatidrivas}}.
  We argue by contradiction. Assume there exists no such constant $C$, then it
  must hold
  \begin{equation}
    \liminf_{\nu \rightarrow 0^+} \nu^{- \frac{\alpha}{\alpha + 2}} 
   r_{\text{dif}} (\nu) = + \infty .
    \label{eq:lower-visc-proof1}
  \end{equation}
  Now take $g_0 (y) = (2 \pi)^{- 1 / 2} e^{i y}$, so that $\|
  g_0 \|_{L^2} = 1 \sim \| g_0 \|_{H^1}$; by
  Definition~\ref{def:diffusion-enhancing} and
  Lemma~\ref{lem:lower-bound-enhanced-1} applied to $\xi = 1$ we deduce that
  there exist constants $C_1, C_2 > 0$ such that, for any $\nu \leqslant 1$
  and $t \geqslant 1$, it holds
  \begin{align*}
    1 - C_1 e^{- r_{\text{dif}} (\nu) t} & \leqslant 1 - \|
    e^{t L_{1, \nu} } \|_{L^2}^2 \leqslant 1 - \| g_t \|_{L^2}^2\\
    & \leqslant C_2 \| g_0 \|^2_{H^1} \left( \nu t + \llbracket u
    \rrbracket_{B^{\alpha}_{1, \infty}} \nu^{\frac{\alpha}{2}} t^{1 +
    \frac{\alpha}{2}} \right)\\
    & \lesssim C_2 (1 + \llbracket u
    \rrbracket_{B^{\alpha}_{1, \infty}}) \nu^{\frac{\alpha}{2}} t^{1 +
    \frac{\alpha}{2}} .
  \end{align*}
  Let $\nu_n \downarrow 0$ be a sequence realizing the liminf
  in~\eqref{eq:lower-visc-proof1} and choose
  \[ t_n = \left( r_{\text{dif}} (\nu_n) \nu_n^{\alpha /
     (\alpha + 2)} \right)^{- 1 / 2} ; \]
  then we obtain
  \[ 1 - C_1 \exp \left( - \left( \nu_n^{- \frac{\alpha}{\alpha + 2}} 
     r_{\text{dif}} (\nu) \right)^{1 / 2} \right)
     \lesssim_{u} \left( \nu_n^{- \frac{\alpha}{\alpha + 2}} 
     r_{\text{dif}} (\nu) \right)^{- \frac{\alpha + 2}{4}} . \]
  Taking the limit as $n \rightarrow \infty$ on both sides we find $1
  \leqslant 0$ which is absurd.
\end{proof}

\subsection{Wei's irregularity condition}\label{sec:wei-enhanced}

A major role in the analysis of dissipation enhancement by rough shear flows
is played by the following condition, first introduced in~{\cite{wei}}.

\begin{definition}
  \label{def:wei-condition}We say that $u \in L^1 (0, T)$ satisfies Wei's
  condition with parameter $\alpha > 0$ if, setting $\psi (y) = \int_0^y u (z)
  \mathd z$, it holds
  \begin{equation}
    \Gamma_{\alpha} (u) \assign \left[ \inf_{\delta \in (0, 1), \bar{y} \in
    [0, T - \delta]} \delta^{- 2 \alpha - 3} \inf_{c_1, c_2 \in \mathbb{R}}
    \int_{\bar{y}}^{\bar{y} + \delta} | \psi (y) - c_1 - c_2 y |^2 \mathd y
    \right]^{1 / 2} > 0. \label{eq:wei-condition}
  \end{equation}
  A similar definition holds for \ $u \in L^1_{\tmop{loc}} (\mathbb{R})$; $u
  \in L^1 (\mathbb{T})$ is said to satisfy Wei's condition once it is
  identified with a $2 \pi$-periodic map on $\mathbb{R}$.
\end{definition}

\begin{remark}
  Denoting by $\mathcal{P}_1$ the set of all polynomials of degree at most
  one, for $u \in L^1_{\tmop{loc}} (\mathbb{R})$ the definition is equivalent
  to
  \[ \Gamma_{\alpha} (u) = \left( \inf_{I \subset \mathbb{R}, | I | < 1} | I
     |^{- 2 \alpha - 3} \inf_{P \in \mathcal{P}_1} \int_I | \psi (y) - P (y)
     |^2 \mathd y \right)^{1 / 2} > 0 ; \]
  this highlights its ``complementarity'' to the seminorm $\llbracket \psi
  \rrbracket_{\mathcal{L}^{2, 2 \alpha + 3}_1}$ associated to the higher order
  Campanato space $\mathcal{L}_1^{2, 2 \alpha + 3}$, as defined
  in~{\cite{campanato1964}}. Observe that $\Gamma_{\alpha}$ is homogeneous,
  i.e. $\Gamma_{\alpha} (\lambda u) = \lambda \Gamma_{\alpha} (u)$ for all
  $\lambda \geqslant 0$.
\end{remark}

The importance of condition~\eqref{eq:wei-condition} comes from the following
result.

\begin{theorem}
  \label{thm:wei}Let $u \in L^1 (\mathbb{T})$ be such that $\Gamma_{\alpha}
  (u) > 0$ for some $\alpha > 0$. Then there exist positive constants $C_1,
  C_2$, depending on $\alpha$ and $\Gamma_{\alpha} (u)$, such that
  \begin{equation}
    \| e^{t L_{k, \nu}} \|_{L^2 \rightarrow L^2} \leqslant C_1 \exp \left( -
    C_2 \nu^{\frac{\alpha}{\alpha + 2}}  | k |^{\frac{2}{\alpha + 2}} t
    \right) \quad \forall \, \nu \in (0, 1), k \in \mathbb{Z}_0, t \geqslant
    0. \label{eq:wei-conclusion}
  \end{equation}
  Namely, $u$ is diffusion enhancing with rate $r_{\tmop{dif}}
  (x) \sim x^{\alpha / (\alpha + 2)}$, in the sense of
  Definition~\ref{def:diffusion-enhancing}.
\end{theorem}

The statement comes from Theorem~5.1 from~{\cite{wei}}; therein $u$ is
required to be continuous, but this restriction is not necessary, see
Appendix~\ref{app:wei-extension} for the proof.

Following the same approach as in Section~\ref{sec:inviscid-mixing}, we
proceed to show that the condition $\Gamma_{\alpha} (u)$ implies irregularity
of $u$; we start by relating it to the property of $\alpha$-H\"{o}lder
roughness, in the sense of Definition~\ref{def:holder-roughness}.

\begin{lemma}
  Let $u \in L^1 (\mathbb{T})$ be such that $\Gamma_{\alpha} (u) > 0$ for some
  $\alpha > 0$. Then $u$ is $\alpha$\mbox{-}H\"{o}lder rough and it holds
  $L_{\alpha} (u) \geqslant \Gamma_{\alpha} (u)$.
\end{lemma}

\begin{proof}
  Fix $\delta > 0$, $\bar{y} \in [- \pi, \pi]$; it holds
  \begin{align*}
    \inf_{c_1, c_2 \in \mathbb{R}} \int_{\bar{y}}^{\bar{y} + \delta} | \psi
    (y) - c_1 - c_2 y |^2 \mathd y & \leqslant \int_{\bar{y}}^{\bar{y} +
    \delta} | \psi (y) - \psi (\bar{y}) - \psi' (\bar{y}) (y - \bar{y}) |^2
    \mathd y\\
    & \leqslant \int_{\bar{y}}^{\bar{y} + \delta} \left( \int_{\bar{y}}^y | u
    (z) - u (\bar{y}) | \mathd z \right)^2 \mathd y\\
    & \leqslant \delta^{2 \alpha + 3}  \left( \sup_{z \in B_{\delta}
    (\bar{y})} \frac{| u (z) - u (\bar{y}) |}{\delta^{\alpha}} \right)^2 .
  \end{align*}
  As the inequality holds for all $\delta$ and $\bar{y}$, we obtain
  $\Gamma_{\alpha} (u)^2 \leqslant L_{\alpha} (u)^2$ and the conclusion.
\end{proof}

We can also relate Wei's condition to regularity in the Besov--Nikolskii
scales $B^{\alpha}_{1, \infty}$.

\begin{lemma}
  \label{lem:irr-wei}Let $u \in L^1 (\mathbb{T})$ be such that
  $\Gamma_{\alpha} (u) > 0$ for some $\alpha \in (0, 1)$. Then $u$ does not
  belong to $B^{\tilde{\alpha}}_{1, \infty}$ for any $\tilde{\alpha} > \alpha$
  and does not belong to $B^{\alpha}_{1, q}$ for any $q < \infty$.
\end{lemma}

\begin{proof}
  For any $\bar{y} \in [- \pi, \pi]$ and $\delta > 0$ it holds
  \begin{align*}
    \delta^{2 \alpha + 3} \Gamma_{\alpha} (u)^2 & \leqslant
    \int_{\bar{y}}^{\bar{y} + \delta} \left| \int_{\bar{y}}^y [u (z) - u
    (\bar{y})] \mathd z \right|^2 \mathd y\\
    & \leqslant \int_{\bar{y}}^{\bar{y} + \delta} \left(
    \int_{\bar{y}}^{\bar{y} + \delta} | u (z) - u (\bar{y}) | \mathd z
    \right)^2 \mathd y
  \end{align*}
  thus implying that
  \begin{equation}
    \inf_{\bar{y} \in \mathbb{T}} \int_0^{\delta} | u (\bar{y} + h) - u
    (\bar{y}) | \mathd h \geqslant \delta^{1 + \alpha} \Gamma_{\alpha} (u)
    \qquad \forall \, \delta \in (0, 1) . \label{eq:irr-wei-proof1}
  \end{equation}
  Now fix $\tilde{\alpha} > \alpha$; starting from~\eqref{eq:irr-wei-proof1}
  and arguing as in the proof of
  Proposition~\ref{prop:rho-irr-besov-nikolskii} (with
  $\varepsilon$ replaced by $\delta$), one obtains
  \[ 2 \pi \Gamma_{\alpha} (u) \leqslant \delta^{\tilde{\alpha} - \alpha}
     \llbracket u \rrbracket_{B^{\tilde{\alpha}}_{1, \infty}}, \]
  which implies the first claim by letting $\delta \rightarrow 0^+$.
  Integrating~\eqref{eq:irr-wei-proof1} over $\bar{y} \in \mathbb{T}$ yields
  \begin{equation}
    \int_0^{\delta} \left\| u \left( \cdot \, + h \right) - u (\cdot)
    \right\|_{L^1} \mathd h \geqslant \delta^{1 + \alpha} \Gamma_{\alpha} (u)
    \quad \forall \, \delta \in (0, 1) ; \label{eq:irr-wei-proof2}
  \end{equation}
  now assume by contradiction that $u \in B^{\alpha}_{1, q}$ for some $q <
  \infty$, then by its equivalent characterization (see
  Appendix~\ref{app:besov}) and the uniform integrability of $h \mapsto h^{- 1
  - \alpha q} \left\| u \left( \cdot \, + h \right) - u (\cdot)
  \right\|^q_{L^1}$ it must hold
  \begin{equation}
    \lim_{\delta \rightarrow 0^+} \int_0^{\delta} \frac{\left\| u \left( \cdot
    \, + h \right) - u (\cdot) \right\|_{L^1}^q}{| h |^{1 + \alpha q}} \mathd
    h = 0. \label{eq:irr-wei-proof3}
  \end{equation}
  On the other hand, by estimate~\eqref{eq:irr-wei-proof2} and Jensen's
  inequality, it holds  
  \begin{align*}
    \int_0^{\delta} \frac{\left\| u \left( \cdot \, + h \right) - u (\cdot)
    \right\|_{L^1}^q}{| h |^{1 + \alpha q}} \mathd h & \geqslant \delta^{- 1 -
    \alpha q} \int_0^{\delta} \left\| u \left( \cdot \, + h \right) - u
    (\cdot) \right\|_{L^1}^q \mathd h\\
    & \geqslant \delta^{- q (1 + \alpha)} \left( \int_0^{\delta} \left\| u
    \left( \cdot \, + h \right) - u (\cdot) \right\|_{L^1} \mathd h
    \right)^q\\
    & \geqslant \Gamma_{\alpha} (u)^q > 0
  \end{align*} 
  uniformly in $\delta \in (0, 1)$, contradicting~\eqref{eq:irr-wei-proof3}.
\end{proof}

\begin{remark}
  It follows from Lemma~\ref{lem:irr-wei} and the construction presented
  Section 2 from~{\cite{colombozelati}} that, for any $\alpha \in \mathbb{Q}$
  as in Lemma~2.1 therein, there exists a Weierstrass-type function which
  belongs to $C^{\alpha} (\mathbb{T})$, satisfies Wei's condition with
  parameter $\alpha$ and does not belong to $B^{\alpha}_{p, q}$ for any $p \in
  [1, \infty], q \in [1, \infty)$, nor to any $B^{\tilde{\alpha}}_{p, q}$ with
  $\tilde{\alpha} > \alpha$.
\end{remark}

In light of Theorem~\ref{thm:wei}, in order to show that almost every shear
flow $u$ enhances dissipation, it will suffice to show that almost every $u$
satisfies Wei's condition. We therefore need to find sufficient conditions in
order for $\Gamma_{\alpha} (u) > 0$ to hold. We start with the following
simple fact, whose proof is almost identical to that of
Lemma~\ref{lem:techlem-rho}, which simplifies the problem by
allowing us to work with not necessarily periodic functions.

\begin{lemma}
  \label{lem:wei-techlem0}A map $u : \mathbb{T} \rightarrow \mathbb{R}$
  satisfies $\Gamma_{\alpha} (u) > 0$ if and only if the maps $u_i : [0, \pi]
  \rightarrow \mathbb{R}$ defined by $u_1 (y) = u
  (y)$, $u_2 (y) = u (-y)$ do
  so.
\end{lemma}

In this way, we can reduce the task to identifying sufficient conditions for
functions defined on a standard interval $[0, \pi]$. For any $\delta > 0$, we
denote by $\Delta^2_{\delta}$ the discrete Laplacian operator
$\Delta^2_{\delta} f (y) = f (y + 2 \delta) - 2 f (y + \delta) + f (y)$.

\begin{lemma}
  \label{lem:wei-techlem-1}For any $\alpha > 0$ and any $(\bar{y}, \delta)$ it
  holds
  \begin{equation}
      \delta^{- 2 \alpha - 3} \inf_{c_1, c_2} \int_{\bar{y}}^{\bar{y} + 3
      \delta} | \psi (y) - c_1 - c_2 y |^2 \mathd y
      \geqslant \frac{1}{12}  \left( \int_{\bar{y}}^{\bar{y} +
      \delta} | \Delta_{\delta}^2 \psi (y) |^{- \frac{1}{1 + \alpha}} \mathd y
      \right)^{- 2 (1 + \alpha)}\label{eq:wei-basic-ineq}
  \end{equation}
  
\end{lemma}

\begin{proof}
  First observe that $\Delta^2_{\delta} (c_1 + c_2 y) \equiv 0$ for any $c_1$,
  $c_2$ and that for any $f$ it holds
  \[ \int_{\bar{y}}^{\bar{y} + 3 \delta} | f (y) |^2 \mathd y \geqslant
     \frac{1}{12} \int_{\bar{y}}^{\bar{y} + \delta} | \Delta^2_{\delta} f (y)
     |^2 \mathd y. \]
  Next, applying Jensen inequality for $g (x) = x^{- \frac{1}{2 (1 +
  \alpha)}}$, which is convex on $(0, \infty)$, it holds
  \[ \left( \frac{1}{\delta} \int_{\bar{y}}^{\bar{y} + \delta} |
     \Delta^2_{\delta} f (y) |^2 \mathd y \right)^{- \frac{1}{2 (1 + \alpha)}}
     \leqslant \frac{1}{\delta} \int_{\bar{y}}^{\bar{y} + \delta} |
     \Delta^2_{\delta} f (y) |^{- \frac{1}{1 + \alpha}} \mathd y. \]
  Algebraic manipulations of the second inequality and the choice $f (y) =
  \psi (y) - c_1 - c_2 y$ yield~\eqref{eq:wei-basic-ineq}.
\end{proof}

In view of Lemma~\ref{lem:wei-techlem-1}, given $\alpha > 0$ and an integrable
$u : [0, \pi] \rightarrow \mathbb{R}$, we define
\begin{equation}
  G_{\alpha} (\bar{y}, \delta) : = \int_{\bar{y}}^{\bar{y} +
  \delta} | \Delta^2_{\delta} \psi (y) |^{- \frac{1}{1 + \alpha}} \mathd y,
  \label{eq:def-G}
\end{equation}
where $\psi$ is defined as usual by $\psi (y) = \int_0^y u (z)
\mathd z$.

\begin{lemma}
  \label{lem:wei-techlem-2}For any $\alpha \in (0, 1)$ and $\varepsilon > 0$,
  define $\beta : = \alpha + \varepsilon (1 + \alpha)$ and
  \[ K_{\alpha, \varepsilon} (u) : = \sup_{n \in \mathbb{N}, 1 \leqslant k
     \leqslant 2^n - 1} 2^{- n \varepsilon} G_{\alpha} (\pi k 2^{- n}, \pi
     2^{- n - 1}) . \]
  Then there exists a constant $C = C (\alpha, \varepsilon)$ such that
  \[ \Gamma_{\beta} (u) \geqslant C (K_{\alpha, \varepsilon} (u))^{- 1 -
     \alpha} . \]
\end{lemma}

\begin{proof}
  First observe that, for any $\beta \in (0, 1),$
  \[ | \Gamma_{\beta} (u) |^2 \sim_{\beta} \inf_{\delta \in (0, 1 / 3),
     \bar{y} \in [0, 1 - 3 \delta]} \delta^{- 2 \beta - 3} \inf_{c_1, c_2 \in
     \mathbb{R}} \int_{\bar{y}}^{\bar{y} + 3 \delta} | \psi (y) - c_1 - c_2 y
     |^2 \mathd y \]
  so to conclude it suffices to provide a lower bound on the latter for our
  choice of $\beta$. Fix $(\bar{y}, \delta)$ and choose $n \in \mathbb{N}$ and
  $k \in \{ 1, \ldots, 2^n - 1 \}$ such that
  \[ \delta \in (\pi 2^{- n}, \pi 2^{- n + 1}], \quad \bar{y} \in [\pi (k - 1)
     2^{- n}, \pi k 2^{- n}] \]
  so that $[\bar{y}, \bar{y} + 3 \delta] \supseteq [\tilde{y}, \tilde{y} + 3
  \tilde{\delta}]$ for the choice $\tilde{y} = \pi k 2^{- n}$, $\tilde{\delta}
  = \pi 2^{- n - 1}$. As a consequence,
  \[ \begin{array}{l}
       \delta^{- 2 \beta - 3} \inf_{c_1, c_2 \in \mathbb{R}}
       \int_{\bar{y}}^{\bar{y} + 3 \delta} | \psi (y) - c_1 - c_2 y |^2 \mathd
       y\\
       \qquad \gtrsim_{\beta} \tilde{\delta}^{- 2 \beta - 3} \inf_{c_1, c_2
       \in \mathbb{R}} \int_{\tilde{y}}^{\tilde{y} + 3 \tilde{\delta}} | \psi
       (y) - c_1 - c_2 y |^2 \mathd y\\
       \qquad \gtrsim \, \tilde{\delta}^{- 2 (\beta - \alpha)}  \left(
       \int_{\tilde{y}}^{\tilde{y} + \tilde{\delta}} | \Delta_{\delta}^2 \psi
       (y) |^{- \frac{1}{1 + \alpha}} \mathd y \right)^{- 2 (1 + \alpha)}\\
       \qquad = (\tilde{\delta}^{\varepsilon} G_{\alpha} (\tilde{y},
       \tilde{\delta}))^{- 2 (1 + \alpha)}
     \end{array} \]
  where in the second passage we employed inequality~\eqref{eq:wei-basic-ineq}
  and then the definition of $\beta$. Overall we deduce by the definition of
  $K$ and the choice of $(\tilde{y}, \tilde{\delta})$ that
  \[ \delta^{- 2 \beta - 3} \inf_{c_1, c_2 \in \mathbb{R}}
     \int_{\bar{y}}^{\bar{y} + 3 \delta} | \psi (z) - c_1 - c_2 z |^2 \mathd z
     \gtrsim_{\beta} K_{\alpha, \varepsilon} (u)^{- 2 (1 + \alpha)} ; \]
  taking the infimum over $(\delta, y)$ then yields the conclusion.
\end{proof}

\subsection{Sufficient conditions for stochastic
processes}\label{sec:enhanced-stochastic}

In order to establish prevalence statements, we want to run the
same programme as in Section~\ref{sec:prevalence-mixing}, exploiting the
properties of LND Gaussian processes and their fundamental translation
invariance from Lemma~\ref{lem:LND-translation}. In order for this strategy to
work, we need an equivalent of
Proposition~\ref{thm:prevalence-rho-irregularity}; this is precisely the aim
of this section, cf. Corollary~\ref{cor:gauss-wei} below. Its proof requires a
few preparations; we start with the following intermediate, general result.

\begin{proposition}
  \label{prop:sufficient-stochastic}Let $u : [0, \pi] \rightarrow \mathbb{R}$
  be an integrable stochastic process, $\psi = \int_0^{\cdot} u_s \mathd s$
  and suppose that there exist $\lambda, \kappa > 0$, $\alpha \in (0, 1)$ such
  that
  \[ \sup_{\delta \in (0, 1), \bar{y} \in [0, \pi - \delta]} \mathbb{E} [\exp
     (\lambda G_{\alpha} (\bar{y}, \delta))] \leqslant \kappa \]
  for $G$ as defined in~\eqref{eq:def-G}. Then for any $\beta > \alpha$ it
  holds $\mathbb{P} (\Gamma_{\beta} (u) > 0) = 1$.
\end{proposition}

\begin{proof}
  By virtue of Lemma~\ref{lem:wei-techlem-2}, for $\beta = \alpha +
  \varepsilon (1 + \alpha)$ it holds
  \[ \mathbb{P} (\Gamma_{\beta} (u) > 0) \geqslant \mathbb{P} (K_{\alpha,
     \varepsilon} (u) < \infty), \]
  so to conclude it suffices to show that $\mathbb{P} (K_{\alpha, \varepsilon}
  (u) < \infty) = 1$ for all $\varepsilon > 0$. Given $\lambda$ as in the
  hypothesis, define the random variable
  \[ J : = \sum_{n \in \mathbb{N}} 2^{- 2 n} \sum_{k = 1}^{2^n - 1} \exp
     (\lambda G_{\alpha} (\pi k 2^{- n}, \pi 2^{- n - 1})) . \]
  By assumption $\mathbb{E} [J] < \infty$, so that $\mathbb{P} (J < \infty) =
  1$. For any $n, k$ it holds
  \[ G_{\alpha} (\pi k 2^{- n}, \pi 2^{- n - 1}) \leqslant \frac{1}{\lambda}
     \log (2^{2 n} J) \lesssim \frac{n}{\lambda}  (1 + \log J) \]
  which implies that
  \[ Y \assign \sup_{n \in \mathbb{N}, 1 \leqslant k \leqslant 2^{- n} - 1}
     \frac{1}{n} G_{\alpha} (\pi k 2^{- n}, \pi 2^{- n - 1}) \lesssim
     \frac{1}{\lambda} (1 + \log J) < \infty \quad \mathbb{P} \text{-a.s.} \]
  Finally, for any $\varepsilon > 0$ it holds $K_{\alpha, \varepsilon} (u)
  \lesssim_{\varepsilon} Y$, which yields the conclusion.
\end{proof}

In order to apply Proposition~\ref{prop:sufficient-stochastic}
to suitable LND Gaussian processes, we will need the three
Lemmas~\ref{lem:gauss-lemma1}-\ref{lem:gauss-lemma3} below.

The next elementary lemma often appears in the probabilistic literature in
connection to so called Krylov or Khasminskii type of estimates, see Lemma~1.1
from~{\cite{portenko1990}} for a slightly more general statement. For the sake
of completeness, we give the proof.

\begin{lemma}
  \label{lem:gauss-lemma1}Let $X$ be a real valued, nonnegative stochastic
  process, defined on an interval $[t_1, t_2]$, adapted to a filtration $\{
  \mathcal{F}_s \}_{s \in [t_1, t_2]}$; suppose there exists a deterministic
  $C > 0$ such that
  \[ \tmop{ess} \sup_{\omega \in \Omega} \mathbb{E} \left[ \int_s^t X_r |
     \mathcal{F}_s \right] \leqslant C \quad \forall \, s \in [t_1, t_2] . \]
  Then for any $\lambda \in (0, 1)$ it holds
  \[ \mathbb{E} \left[ \exp \left( \frac{\lambda}{C} \int_{t_1}^{t_2} X_r
     \mathd r \right) \right] \leqslant (1 - \lambda)^{- 1} . \]
\end{lemma}

\begin{proof}
  Up to rescaling $X$, we may assume $C = 1$. It holds
  \[ \mathbb{E} \left[ \exp \left( \lambda \int_{t_1}^{t_2} X_r \mathd r
     \right) \right] = \sum_{n = 0}^{\infty} \frac{\lambda^n}{n!} \mathbb{E}
     \left[ \left( \int_{t_1}^{t_2} X_r \mathd r \right)^n \right] = \sum_{n =
     0}^{\infty} \lambda^n I_n \]
  where
  \[ I_n =\mathbb{E} \left[ \int_{t_1 < r_1 < \ldots < r_n < t_2} X_{r_1}
     \cdot \ldots \cdot X_{r_n} \mathd r_1 \ldots \mathd r_n \right] . \]
  By the assumptions and the non-negativity of $X$, it holds
  \begin{align*}
    I_n & = \int_{t_1 < r_1 < \ldots < r_{n - 1} < t_2} \mathbb{E} \left[
    X_{r_1} \cdot \ldots \cdot X_{r_{n - 1}} \int_{r_{n - 1}}^t X_{r_n} \mathd
    r_n \right] \mathd r_1 \ldots \mathd r_{n - 1}\\
    & = \int_{t_1 < r_1 < \ldots < r_{n - 1} < t_2} \mathbb{E} \left[ X_{r_1}
    \cdot \ldots \cdot X_{r_{n - 1}} \mathbb{E} \left[ \int_{r_{n - 1}}^t
    X_{r_n} \mathd r_n | \mathcal{F}_{r_{n - 1}} \right] \right] \mathd r_1
    \ldots \mathd r_{n - 1}\\
    & \leqslant \int_{t_1 < r_1 < \ldots < r_{n - 1} < t_2} \mathbb{E}
    [X_{r_1} \cdot \ldots \cdot X_{r_{n - 1}}] \mathd r_1 \ldots \mathd r_{n -
    1} = I_{n - 1}
  \end{align*}
  which iteratively implies $I_n \leqslant 1$. Therefore we obtain
  \[ \mathbb{E} \left[ \exp \left( \lambda \int_s^t X_u \mathd u \right)
     \right] \leqslant \sum_{n = 0}^{\infty} \lambda^n = (1 - \lambda)^{- 1} . \qedhere
  \]
\end{proof}

\begin{lemma}
  \label{lem:gauss-lemma2}Let $Z \sim \mathcal{N} (m, \sigma^2)$ be a real
  valued Gaussian variable. Then for any $\theta \in (0, 1)$ there exists
  $c_{\theta} > 0$ such that
  \[ \mathbb{E} [| Z |^{- \theta}] \leqslant c_{\theta} \sigma^{- \theta} . \]
\end{lemma}

\begin{proof}
  Set $Z = \sigma N + m$, then $\mathbb{E} [| Z |^{- \theta}] = \sigma^{-
  \theta} \mathbb{E} [| N - x |^{- \theta}]$ for $x = - m / \sigma$; therefore
  is sufficed to show that
  \[ \sup_{x \in \mathbb{R}} \mathbb{E} [| N - x |^{- \theta} ] = \sup_{x \in
     \mathbb{R}} \int | x - y |^{- \theta} p (y) \mathd y = \| | \cdot |^{-
     \theta} \ast p \|_{L^{\infty}} < \infty \]
  where $p$ stands for the Gaussian density $p (x) = (2 \pi)^{- 1 / 2} \exp (-
  | x |^2 / 2)$. By Young's inequality it holds
  \begin{align*}
    \| | \cdot |^{- \theta} \ast p \|_{L^{\infty}} & \leqslant \| (| \cdot
    |^{- \theta} \mathbbm{1}_{| \cdot | < 1}) \ast p \|_{L^{\infty}} + \| (|
    \cdot |^{- \theta} \mathbbm{1}_{| \cdot | \geqslant 1}) \ast p
    \|_{L^{\infty}}\\
    & \leqslant \| | \cdot |^{- \theta} \mathbbm{1}_{| \cdot | < 1} \|_{L^1}
    \| p \|_{L^{\infty}} + \| | \cdot |^{- \theta} \mathbbm{1}_{| \cdot |
    \geqslant 1} \|_{L^{\infty}}  \| p \|_{L^1}\\
    & \leqslant (2 \pi)^{- 1 / 2} \| | \cdot |^{- \theta} \mathbbm{1}_{|
    \cdot | < 1} \|_{L^1} + 1 < \infty
  \end{align*}
  which gives the conclusion.
\end{proof}

\begin{lemma}
  \label{lem:gauss-lemma3}Let $Y : [0, \pi] \rightarrow \mathbb{R}$ be a $(1 +
  H)$-SLND Gaussian process with constant $C_Y$, in the sense of
  Definition~\ref{def:LND}. Then for any $\alpha > H$ there exists $\lambda =
  \lambda (\alpha, H, C_Y) > 0$ s.t.
  \[ \mathbb{E} \left[ \exp \left( \lambda \int_{\bar{y}}^{\bar{y} + \delta} |
     \Delta^2_{\delta} Y_y |^{- \frac{1}{1 + \alpha}} \mathd y \right) \right]
     \leqslant 2 \quad \forall \, \delta \in (0, 1), \bar{y} \in [0, \pi -
     \delta] . \]
\end{lemma}

\begin{proof}
  The result follows Lemmas~\ref{lem:gauss-lemma1} and~\ref{lem:gauss-lemma2}
  applied to the process $X_y = | \Delta^2_{\delta} \psi_y |^{- \frac{1}{1 +
  \alpha}}$. Indeed, denote by $\mathcal{F}_y$ the natural filtration
  generated by $\psi$ and let $\mathcal{G}_y \assign \mathcal{F}_{y + 2
  \delta}$. It is clear that $\Delta^2_{\delta} \psi_y = Y_{y + 2 \delta} - 2
  Y_{y + \delta} + Y_y$ is $\mathcal{G}_y$-adapted; for any $[z, y] \subset
  [\bar{y}, \bar{y} + \delta]$ it holds
  \[ \tmop{Var} (\Delta^2_{\delta} Y_y | \mathcal{G}_z) = \tmop{Var} (Y_{y + 2
     \delta} | \mathcal{F}_{z + 2 \delta}) \geqslant C_Y | y - z |^{2 (1 + H)}
     . \]
  As a consequence, we have a decomposition $\Delta^2_{\delta} Y_y =
  Z^{(1)}_{z, y} + Z^{(2)}_{z, y}$ with $Z^{(1)}_{z, y}$ adapted to
  $\mathcal{G}_z$ and $Z^{(2)}_{z, y}$ Gaussian and independent of
  $\mathcal{G}_z$; therefore
  \begin{align*}
    \mathbb{E} \left[ \int_u^{\bar{y} + \delta} | \Delta^2_{\delta} Y_y |^{-
    \frac{1}{1 + \alpha}} \mathd y | \mathcal{G}_z \right] & = \int_z^{\bar{y}
    + \delta} \mathbb{E} \left[ | Z^{(2)}_{z, y} + \cdot |^{- \frac{1}{1 +
    \alpha}} \right] (Z^{(1)}_{u, y}) \mathd y.
  \end{align*}
  By Lemma~\ref{lem:gauss-lemma2}, since $\tmop{Var} (Z^{(2)}_{z, y})
  \geqslant C_Y {| y - z |^{2 (1 + H)}} $ and $\theta = (1 + \alpha)^{- 1} \in
  (0, 1)$, it holds
  \[ \sup_{x \in \mathbb{R}} \mathbb{E} \left[ | Z^{(2)}_{z, y} + x |^{-
     \frac{1}{1 + \alpha}} \right] \lesssim_{\alpha} \tmop{Var} (Z^{(2)}_{z,
     y})^{- \frac{1}{2 (1 + \alpha)}} \lesssim_{\alpha, H, C_Y} | y - z |^{-
     \frac{1 + H}{1 + \alpha}} \]
  and thus
  \begin{align*}
    \mathbb{E} \left[ \int_z^{\bar{y} + \delta} | \Delta^2_{\delta} X_y |^{-
    \frac{1}{1 + \alpha}} \mathd y | \mathcal{G}_z \right] & \lesssim
    \int_z^{\bar{y} + \delta} | y - z |^{- \frac{1 + H}{1 + \alpha}} \mathd
    z\\
    & \lesssim \int_0^1 | r |^{- \frac{1 + H}{1 + \alpha}} \mathd r \sim C
    (\alpha, H, C_Y)
  \end{align*}
  where the estimate is uniform over $z \in [\bar{y}, \bar{y} + \delta]$,
  $\bar{y} \in \mathbb{T}$ and $\delta \in (0, 1)$. Choosing
  \[ \lambda = \frac{1}{2 C (\alpha, H, C_Y)}, \]
  we obtain the conclusion by applying Lemma~\ref{lem:gauss-lemma1}.
\end{proof}

With Lemmas~\ref{lem:gauss-lemma1}\mbox{-}\ref{lem:gauss-lemma3}
at hand, we can finally verify that suitable Gaussian processes verify Wei's
condition with probability~$1$; we give the statement in full generality, but
we stress that the most relevant example verifying the hypothesis below is the
canonical process $X$ associated to $\mu^H$, as granted by
Lemma~\ref{lem:fBm-LND}.

\begin{corollary}
  \label{cor:gauss-wei}Let $X : [0, \pi] \rightarrow \mathbb{R}$ be a Gaussian
  process such that
  \[ Y_y = \int_0^y X_z \mathd z  \]
  is $(1 + H)$\mbox{-}SLND for some $H \in (0, 1)$. Then
  \[ \mathbb{P} (\Gamma_{\alpha} (X) > 0) = 1 \]
  for any $\alpha > H$.
\end{corollary}

\begin{proof}
  It follows immediately combining Lemma~\ref{lem:gauss-lemma3} and
  Proposition~\ref{prop:sufficient-stochastic}.
\end{proof}

\subsection{Prevalence statements and proof of
Theorems~\ref{thm:main-thm-enhanced},~\ref{thm:main-thm-1}}\label{sec:prevalence-enhanced}

Similarly to Section~\ref{sec:prevalence-mixing}, in order to
prove prevalence statements in $B^{\alpha}_{1, \infty} (\mathbb{T})$, we will
actually start by establishing their analogues on $B^{\alpha}_{1, \infty} (0,
\pi)$.

\begin{theorem}
  \label{thm:prevalence-interval-wei}Let $\alpha \in (0, 1)$; then a.e.
  $\varphi \in B^{\alpha}_{1, \infty} (0, \pi)$ satisfies $\Gamma_{\beta}
  (\varphi) > 0$ for all $\beta > \alpha$.
\end{theorem}

\begin{proof}
  Fix $\alpha \in (0, 1)$ and define $\mathcal{A} : = \left\{ \varphi \in
  B^{\alpha}_{1, \infty} (0, \pi) : \, \Gamma_{\beta} (\varphi) > 0 \text{ for
  all } \beta > \alpha \right\}$; it holds
  \[ \mathcal{A} = \bigcap_{n = 1}^{\infty} \bigcup_{m = 1}^{\infty}
     \mathcal{A}_{n, m} \assign \bigcap_{n = 1}^{\infty} \bigcup_{m =
     1}^{\infty} \left\{ \varphi \in B^{\alpha}_{1, \infty} (0, \pi) :
     \Gamma_{\beta} (\varphi) \geqslant \frac{1}{m} \text{ for } \beta =
     \alpha + \frac{1}{n} \right\} . \]
  The sets $\mathcal{A}_{n, m}$ are closed in the topology of $B^{\alpha}_{1,
  \infty} (0, \pi)$ (the map $\varphi \mapsto \Gamma_{\beta} (\varphi)$ is
  upper semicontinuous in the topology of $L^1 (0, \pi)$), thus $\mathcal{A}$
  is Borel measurable. In order to conclude, it is enough to show that for any
  fixed $\beta > \alpha$, the set $\mathcal{A}_{\beta} : = \left\{ \varphi \in
  B^{\alpha}_{1, \infty} (0, \pi) : \, \Gamma_{\beta} (\varphi) > 0 \right\}$
  (which is Borel by the same line of argument) is prevalent.
  
  Now fix $\beta > \alpha$ and choose $H \in (\alpha, \beta)$; denote by
  $\mu^H$ the law of fBm of parameter $H$ on $C ([0, \pi])$ and by $u = \{ u_y
  \}_{y \in [0, \pi]}$ the associated canonical process. Since $\mu^H$ is
  supported on $C^{H - \varepsilon} ([0, \pi])$ for any $\varepsilon > 0$ and
  $H > \alpha$, it is also a tight probability measure on $B^{\alpha}_{1,
  \infty} (0, \pi)$. By Lemma~\ref{lem:fBm-LND}, the associated process $\psi
  = \int_0^{\cdot} u (y) \mathd y$ is $(1 + H)$-SLND and so by
  Lemma~\ref{lem:LND-translation} the same holds for $f + \psi$, for any
  measurable $f : [0, \pi] \rightarrow \mathbb{R}$.
  
  In particular, for a given $\varphi \in B^{\alpha}_{1, \infty} (0, \pi)$,
  taking $f = \int_0^{\cdot} \varphi (y) \mathd y$, it follows from
  Corollary~\ref{cor:gauss-wei} and the choice $\beta > H$ that
  
  \begin{align*}
    \mu^H (\varphi + \mathcal{A}_{\beta}) & = \mu^H \left( \left\{ u \in
    B^{\alpha}_{1, \infty} (0, \pi) : \, \Gamma_{\beta} (u + \varphi) > 0
    \right\} \right) = 1.
  \end{align*}
  
  As the reasoning holds for any $\varphi \in B^{\alpha}_{1, \infty} (0,
  \pi)$, we deduce that $\mu^H$ witnesses the prevalence of
  $\mathcal{A}_{\beta}$ and we obtain the conclusion.
\end{proof}

As in Section~\ref{sec:prevalence-mixing}, we define for $\tilde{\varphi} :
[0, \pi] \rightarrow \mathbb{R}$ the map $(T \tilde{\varphi}) (y) =
\tilde{\varphi} (| y |)$; conversely for $\varphi : \mathbb{T} \rightarrow
\mathbb{R}$, $\varphi_1 (y) : = \varphi (y)$, $\varphi_2 (y) : = \varphi (-
y)$. Recall that if $\tilde{\varphi} \in B^{\alpha}_{1, \infty} \cap
L^{\infty}$, then $T \tilde{\varphi} \in B^{\alpha}_{1, \infty}$.

\begin{corollary}
  \label{cor:prevalence-periodic-wei}Almost every $\varphi \in B^{\alpha}_{1,
  \infty} (\mathbb{T})$ satisfies $\Gamma_{\beta} (\varphi) > 0$ for all
  $\beta > \alpha$.
\end{corollary}

\begin{proof}
  The proof is very similar to that of
  Corollary~\ref{cor:prevalence-periodic-inviscid}, again employing measures
  of the form $\nu^H = T_{\sharp} \mu^H$ for suitable $H \in (0, 1)$;
  specifically, once we fix $\beta > \alpha$ and we define a subset
  $\mathcal{A}_{\beta}$ of $B^{\alpha}_{1, \infty} (\mathbb{T})$ as in the
  proof of Theorem~\ref{thm:prevalence-interval-wei}, it suffices to choose $H
  \in (\alpha, \beta)$. In this way $\mu^H$ is tight on $B^{\alpha}_{1,
  \infty} (0, \pi) \cap L^{\infty} (0, \pi)$, so $\nu^H$ is tight on
  $B^{\alpha}_{1, \infty} (\mathbb{T})$; the verification that
  \[ \nu^H (\varphi + \mathcal{A}_{\beta}) = 1 \quad \forall \, \varphi \in
     B^{\alpha}_{1, \infty} (\mathbb{T}) \]
  is almost identical to that of
  Corollary~\ref{cor:prevalence-periodic-inviscid}, only this time invoking
  Lemma~\ref{lem:wei-techlem0} and Theorem~\ref{thm:prevalence-interval-wei}.
\end{proof}

At this point we have all the ingredient to close the dissipative case.

\begin{proof}[Proof of Theorem~\ref{thm:main-thm-enhanced}]
  The lower bound comes from Proposition~\ref{prop:lower-bound-enhanced},
  while the upper bound from a combination of Theorem~\ref{thm:wei} and
  Corollary~\ref{cor:prevalence-periodic-wei}.
\end{proof}

The main result of the paper, Theorem~\ref{thm:main-thm-1}, is now a direct
consequence of Theorems~$\ref{thm:main-thm-inviscid}$
and~\ref{thm:main-thm-enhanced}.
In fact, let us record here a slightly sharper estimate. Given $f \in L^2
(\mathbb{T}^2)$, for any $s \in \mathbb{R}$ define
\[ \| f \|_{L^2_x H^s_y}^2 \assign \sum_{k \in \mathbb{Z}} \| P_k f \|_{H^s
   (\mathbb{T}; \mathbb{C})}^2 = \sum_{(k, \eta) \in \mathbb{Z}^2} (1 + | \eta
   |^2)^s  | \hat{f} (k, \eta) |^2 ; \]
it's clear that, for $s \geqslant 0$, $\| f \|_{L^2_x H^s_y} \leqslant \| f
\|_{H^s (\mathbb{T}^2)}$ and $\| f \|_{H^{- s} (\mathbb{T}^2)} \leqslant \| f
\|_{L^2_x H^{- s}_y}$.

\begin{theorem}
  Almost every $u \in B^{\alpha}_{1, \infty} (\mathbb{T})$ satisfies the
  following property: for any $\tilde{\alpha} > \alpha$, there exists $C = C
  (\alpha, \tilde{\alpha}, u)$ such that, for any $f_0 \in H^{1 / 2}
  (\mathbb{T}^2)$ with $P_0 f_0 \equiv 0$, it holds
  \begin{equation}
    \| e^{t u \partial_x} f_0 \|_{{L^2_x}  H^{- 1 / 2}_y} \leqslant C t^{-
    \frac{1}{2 \tilde{\alpha}}} \| f_0 \|_{L^2_x H^{1 / 2}_y} .
    \label{eq:final-thm-eq1}
  \end{equation}
  Almost every $u \in B^{\alpha}_{1, \infty} (\mathbb{T})$ satisfies the
  following property: for any $\tilde{\alpha} > \alpha$ there exist $C_i = C
  (\alpha, \tilde{\alpha}, u)$ such that, for any $f_0 \in L^2 (\mathbb{T}^2)$
  with $P_0 f_0 \equiv 0$, it holds
  \begin{equation}
    \| e^{- t (u \partial_x - \nu \Delta)} f \|_{L^2 (\mathbb{T}^2)} \leqslant
    C_1 \exp \left( - C_2 t \nu^{\frac{\tilde{\alpha}}{\tilde{\alpha} + 2}}
    \right) \| f_0 \|_{L^2 (\mathbb{T}^2)} . \label{eq:final-thm-eq2}
  \end{equation}
\end{theorem}

\begin{proof}
  By Theorem~\ref{thm:main-thm-inviscid} b), for almost every $u \in
  B^{\alpha}_{1, \infty} (\mathbb{T})$ and any $\tilde{\alpha} > \alpha$ it
  holds
  \begin{align*}
    \| e^{- t u \partial_x} f_0 \|_{{L^2_x}  H^{- 1 / 2}_y}^2 & = \sum_{k \in
    \mathbb{Z}_0} \| P_k (e^{- t u \partial_x} f_0) \|_{H^{- 1 / 2}}^2
    \lesssim \sum_{k \in \mathbb{Z}_0} (t | k |)^{- \frac{1}{\tilde{\alpha}}}
    \| P_k f_0 \|_{H^{- 1 / 2}}^2 \lesssim t^{- \frac{1}{\tilde{\alpha}}} \|
    f_0 \|_{{L^2_x}  H^{- 1 / 2}_y}^2
  \end{align*}
  proving~\eqref{eq:final-thm-eq1}. Denote $\mathcal{L}_{\nu} = - u \partial_x
  + \nu \partial_y^2$, so that $- u \partial_x + - \nu \Delta =
  \mathcal{L}_{\nu} + \nu \partial_x^2$, where the operators
  $\mathcal{L}_{\nu}$ and $\nu \partial_x^2$ commute; also observe that $P_k
  (e^{t \partial_x^2} f) = e^{- t k^2} P_k  f$.
  
  Combining these properties with Theorem~\ref{thm:main-thm-enhanced}, for
  almost every $u \in B^{\alpha}_{1, \infty} (\mathbb{T})$ and any
  $\tilde{\alpha} > \alpha$ it holds
  \begin{align*}
    \| e^{t (- u \partial_x + \nu \Delta)} f \|_{L^2}^2 & = \sum_{k \in
    \mathbb{Z}_0} \| P_k (e^{t \partial_x^2} e^{t \mathcal{L}_{\nu}} f)
    \|_{L^2}^2 = \sum_{k \in \mathbb{Z}_0} e^{- 2 t k^2} \| P_k (e^{t
    \mathcal{L}_{\nu}} f) \|_{L^2}^2\\
    & \lesssim \sum_{k \in \mathbb{Z}_0} \exp \left( - 2 t | k |^2 - C t
    \nu^{\frac{\alpha}{\alpha + 2}} | k |^{\frac{2}{\alpha + 2}} \right) \|
    P_k f \|_{L^2}^2\\
    & \lesssim \exp \left( - C t \nu^{\frac{\alpha}{\alpha + 2}} \right)
    \sum_{k \in \mathbb{Z}_0} \| P_k f \|_{L^2}^2
  \end{align*}
  which yields~\eqref{eq:final-thm-eq2}.
\end{proof}

\section{Further comments and future directions}

We have shown in this paper that generic rough shear flows satisfy both
inviscid mixing and enhanced dissipation properties, with rates sharply
determined by the regularity parameter $\alpha \in (0, 1)$ in the Besov scale
$B^{\alpha}_{1, \infty}$. In the enhanced dissipation case, this confirms the
intuition from~{\cite{colombozelati}}; instead in the inviscid mixing one, it
shows that the behavior presented by Weierstrass-type functions constructed
therein is not generic in the sense of prevalence. Our results provide a
connection to the property of $\rho$-irregularity, which was never observed in
this context, and highlight the importance of working with mixing scales $H^{-
s}$ with $s \neq 1$.

\

We conclude by presenting a few additional remarks and open problems arising
from this work.
\begin{enumeratenumeric}
  \item We are currently unable to determine whether there is a clear
  connection between the properties of $\rho$-irregularity and Wei's
  condition. Lemma~\ref{lem:estim-mixing-rho}, together with the trivial
  estimate $\| f_t \|_{H^{- 1}} \leqslant \| f_t \|_{H^{- 1 / 2}}$, imply that
  for $\alpha \in (0, 1 / 2)$ the shear flows $u \in C^{\alpha}$ constructed
  in~{\cite{colombozelati}} satisfy $\Gamma_{\alpha} (u) > 0$ but are not
  $\rho$-irregular with $\rho \sim (2 \alpha)^{- 1}$. Heuristically, this fact
  is similar to the existence of flows with small dissipation time which are
  not mixing, like the cellular flows presented in~{\cite{iyer2021}}.
  
  \item The above argument also implies the existence of Weierstrass type
  functions which are not $\rho$-irregular, for suitable values $\rho$. We
  believe this problem was open in the probabilistic community, although never
  been explicitly addressed in the literature.
  
  \item Even without establishing a direct connection to Wei's condition, it
  would be interesting to show that functions $u$ which are $\rho$-irregular
  for $\rho \sim (2 \alpha)^{- 1}$ are diffusion enhancing with rate
  $r_{\text{dif}} (\nu) \sim \nu^{\alpha / (\alpha + 2)}$, in
  line with the heuristic argument presenting in the introduction. Since such
  $u$ are mixing, they are indeed diffusion enhancing with a suitable rate
  by~{\cite{constantin2008}}; however the quantitative results
  from~{\cite{zelati2020relation}} only imply the worsened rate $\nu^{3 \alpha
  / (1 + 3 \alpha)}$ and it is rather unclear how to ``bridge the gap''.
  
  \item Going through the same proof as in Lemma~\ref{lem:estim-mixing-rho},
  one can show that if $u$ is $(\gamma, \rho)$-irregular with $\gamma > 1 -
  s$, then $u$ is mixing on the scale $H^s$ with rate $r_{s
  \text{-mix}} (t) = t^{\rho}$. In the case $\gamma = 0$ an
  even simpler proof based on duality and integration by parts provides mixing
  on the scale $H^1$ with the same rate. In fact, since $H^1 (\mathbb{T};
  \mathbb{C})$ is an algebra, by integration by parts it holds
  \[ \begin{array}{ll}
       | \langle e^{i \xi u} f, g \rangle | & = \left| \int_{- \pi}^{\pi} e^{i
       \xi u (y)} f (y) g (y) \mathd y \right|\\
       & \leqslant | (f g) (- \pi) | \left| \int_{- \pi}^{\pi} e^{i \xi u
       (y)} \mathd y \right|\\
       & \qquad + \int_{- \pi}^{\pi} | (f g)' (y) | \left| \int_{- \pi}^y
       e^{i \xi u (z)} \mathd z \right| \mathd y\\
       & \lesssim (\| f g \|_{L^{\infty}} + \| (f g)' \|_{L^1})  \| \Phi^u
       \|_{0, \rho}  | \xi |^{- \rho}\\
       & \lesssim \| f \|_{H^1} \| g \|_{H^1} \| \Phi^u \|_{0, \rho}  | \xi
       |^{- \rho}
     \end{array} \]
  which by duality implies $\| e^{i \xi u} f \|_{H^{- 1}} \lesssim \| f
  \|_{H^1} \| \Phi^u \|_{0, \rho}  | \xi |^{- \rho}$ and so the claim.
  
  \item It is however an open problem to provide examples of $(0,
  \rho)$-irregular functions $u$, for any $\rho < 1$. See Remark~69
  in~{\cite{galeati2020prevalence}} for a deeper discussion. There are several
  examples of $u : [0, \pi] \rightarrow \mathbb{R}$ which are $(0,
  1)$-irregular, including the choice $u (y) = y$; by Proposition~1.4
  from~{\cite{choukgubinelli1}}, it is enough to require the existence of
  $\delta > 0$ such that
  \begin{equation}
    \frac{1}{\delta} \leqslant | u' (y) |, \quad \frac{| u'' (y) |}{| u' (y)
    |} \leqslant \delta \quad \forall \, y \in [0, \pi] ;
    \label{eq:condition-rho}
  \end{equation}
  observe the similarity of condition~\eqref{eq:condition-rho} with
  Assumption~(H) from~{\cite{zelati2020stable}}.
  
  \item The property of $(\gamma, \rho)$-irregularity can be reformulated in
  terms of the (Fourier transform of) \tmtextit{occupation measure} of $u$,
  namely the family $\{ \mu^u_I, I \subset \mathbb{T} \}$ given by $\mu^u_I =
  u_{\sharp} \mathcal{L}_I$ where $I$ are subintervals of $\mathbb{T}$ and
  $\mathcal{L}_I$ stands for the Lebesgue measure on $I$; see Section~2.3
  from~{\cite{galeati2020prevalence}} for more details. Closely related to the
  occupation measure of $u$ is its \tmtextit{local time}, namely the
  Radon--Nikodym derivative $\mathd \mu^u_{\mathbb{T}} / \mathd
  \mathcal{L}_{\mathbb{T}}$, which has been intensively studied for stochastic
  processes, see the review~{\cite{geman}}. The following question arises
  naturally: is it possible to link the mixing properties of $u$ to the
  regularity of its local time?
  
  \item In the paper we have always focused on the scales $B^{\alpha}_{1,
  \infty} (\mathbb{T})$ with $\alpha \in (0, 1)$. If one is instead interested
  in the mixing properties of generic $u \in C (\mathbb{T})$, much faster
  rates are available. Indeed for any $\beta > 1$ it's possible to construct
  $\tilde{u}^{\beta} \in C ([0, \pi])$ satisfying
  \begin{equation}
    \left| \int_{y_1}^{y_2} e^{i \xi u^{\beta} (z)} \mathd z \right|
    \lesssim_{\gamma, \beta} | y_2 - y_1 |^{\gamma} \exp \left( - C_{\gamma,
    \beta}  | \xi |^{\frac{2}{1 + \beta}} \right) \quad \forall \, \xi \in
    \mathbb{R}, \quad 0 \leqslant y_1 \leqslant y_2 \leqslant \pi
    \label{eq:exp-rho-irr}
  \end{equation}
  and so by symmetrization the same holds for $u^{\beta} \assign T
  \tilde{u}^{\beta}$. Such $\tilde{u}^{\beta}$ are given by typical
  realization of the so called $(2 \beta)$-log Brownian motion,
  see~{\cite{perkowski}} for its definition and Propositions~48 and~49
  from~{\cite{galeati2020prevalence}} for the proof of~\eqref{eq:exp-rho-irr}.
  In fact, one could use the law of such process to prove that a.e. $u \in C
  (\mathbb{T})$ satisfies~\eqref{eq:exp-rho-irr} for any $\beta > 1$ (the
  value $\beta = 1$ can only be attained by Caratheodory functions, which are
  naturally discontinuous). Arguing as in the proof of
  Lemma~\ref{lem:estim-mixing-rho} one can deduce that such $u$ are
  exponentially mixing, in the sense that they satisfy the estimate
  \begin{equation}
    \| e^{i \xi u} g \|_{H^{- 1}} \lesssim \exp \left( - C_{\gamma, \beta}  |
    \xi |^{\frac{2}{1 + \beta}} \right)  \| g \|_{H^1} \quad \forall \, g \in
    H^1 (\mathbb{T}; \mathbb{C})
  \end{equation}
  and so that
  \begin{equation}
    \| e^{- t u \partial_x} f \|_{L^2_x H^{- 1}_y} \lesssim \exp \left( -
    C_{\gamma, \beta} t^{\frac{2}{1 + \beta}} \right) \| f \|_{L^2_x H^1_y}
  \end{equation}
  for all $f \in H^1 (\mathbb{T}^2)$ satisfying $P_0 f \equiv 0$.
  
  \item Finally, let us point out that the property of $\rho$-irregularity
  also holds for generic vector valued functions $u : [0, 1] \rightarrow
  \mathbb{R}^d$ (resp. $u : \mathbb{T} \rightarrow \mathbb{R}^d$), for any $d
  \in \mathbb{N}$, see~{\cite{galeati2020prevalence}}. In particular, similar
  statements to part \tmtextit{i.} of Theorem~\ref{thm:main-thm-1} can be
  established for ``higher dimensional'' shear flows of the form
  \[ \partial_t f + \bar{u} \cdot \nabla f = \nu \Delta f \]
  for $f : \mathbb{T}^{d + 1} \rightarrow \mathbb{R}$, $\bar{u} (x_1, \ldots,
  x_{d + 1}) \assign (u (x_{d + 1}), 0)^T$; observe that for $d = 2$,
  $\bar{u}$ is a stationary solution to $3$D Euler equations. In light
  of~{\cite{constantin2008}}, the vector field $\bar{u}$ constructed by a
  $\rho$-irregular $u$ is diffusion enhancing; thus can be applied in the
  study of suppression of blow-up by mixing phenomena similarly to what was
  done in~{\cite{kiselev2016,bedrossian2017suppression,iyer2021}}.
\end{enumeratenumeric}
\appendix\section{Besov spaces}\label{app:besov}

In this appendix we record fundamentals on Besov spaces $B^s_{p, q}$ on the
torus $\mathbb{T}^d$, although in the paper we only need the case $d = 1$. For
a gentle introduction on spaces on $\mathbb{R}^d$ we refer to the
monograph~{\cite{bahouri}}; see also the classical paper~{\cite{simon}} for
spaces on an interval $I \subset \mathbb{R}$. All their properties transfer to
the analogous spaces on $\mathbb{T}^d$ by a clever use of Poisson summation
formula, see~{\cite{gubinelli2015lectures}},~{\cite{mourrat}}. Alternatively,
periodic Besov spaces have been treated in Chapter~3
of~{\cite{schmeisser1987}}.

\

Given a dyadic partition of the unity $(\chi, \varphi)$ and the associated
Littlewood--Paley blocks $\{ \Delta_j \}_{j \geqslant - 1}$, the
(inhomogeneous) Besov spaces $B^s_{p, q} (\mathbb{T}^d)$ with $s \in
\mathbb{R}$, $p, q$ are defined as the set of tempered distributions $f \in
\mathcal{S}' (\mathbb{T}^d)$ such that
\[ \| f \|_{B^s_{p, q}} \assign \| 2^{s j}  \| \Delta_j f \|_{L^p} \|_{\ell^q}
   = \left( \sum_{j = - 1}^{\infty} 2^{j s q}  \| \Delta_j f \|_{L^p}^q
   \right)^{1 / q} < \infty \]
with the usual conventions when $q = \infty$. $(B^s_{p, q} (\mathbb{T}^d), \|
\cdot \|_{B^s_{p, q}})$ are Banach spaces and the definition is independent of
the choice of the partition of unity $(\chi, \varphi)$. Besov spaces are handy
to use due to their many properties, including functional embeddings and
behavior under derivation and multiplication; we recall them briefly.

\begin{proposition}[{\cite{bahouri}}, Prop. 2.71]
  \label{prop:besov-embeddings}Let $1 \leqslant p_1 \leqslant p_2 \leqslant
  \infty$ and $1 \leqslant q_1 \leqslant q_2 \leqslant \infty$. Then for any
  $s \in \mathbb{R}$, the space $B^s_{p_1, q_1}$ continuously embeds in $B^{s
  - d \left( \frac{1}{p_1} - \frac{1}{p_2} \right)}_{p_2, q_2}$.
\end{proposition}

Also recall the following basic facts: for any $\varepsilon > 0$ and any $p, q
\in [1, \infty]$, the space $B^s_{p, q}$ continuously embeds in $B^{s -
\varepsilon}_{p, 1}$, as can be checked using the definition; for any $p \in
[1, \infty]$, we have the embeddings
\[ B^{0}_{p, 1} \hookrightarrow L^p \hookrightarrow
   B^{0}_{p, \infty} \]
see for instance Remark~A.3 from~{\cite{mourrat}} for the second statement.

\begin{proposition}[{\cite{mourrat}}, Prop. A.5]
  \label{prop:besov-differentiation}Let $s \in \mathbb{R}$, $p, q \in [1,
  \infty]$, $i \in \{ 1, \ldots, n \}$. Then the map $f \mapsto \partial_i f$
  is a continuous linear operator from $B^s_{p, q}$ to $B^{s - 1}_{p, q}$.
\end{proposition}

\begin{proposition}[{\cite{mourrat}}, Prop. A.7]
  \label{prop:besov-paraproducts}Let $s_1, s_2 \in \mathbb{R}$ and $p, p_1,
  p_2, q \in [1, \infty]$ be such that
  \[ s_1 < 0 < s_2, \quad s_1 + s_2 > 0, \quad \frac{1}{p} = \frac{1}{p_1} +
     \frac{1}{p_2} ; \]
  then $(f, g) \mapsto f g$ is a well-defined continuous bilinear map from
  $B^{s_1}_{p_1, q} \times B^{s_2}_{p_2, q}$ to $B^{s_1}_{p, q}$.
\end{proposition}

\begin{proposition}[{\cite{bahouri}}, Cor. 2.86]
  \label{prop:besov-algebra}For any $s > 0$ and $p, q \in [1, \infty]$, the
  space $B^s_{p, q} \cap L^{\infty}$ is an algebra and there exists a constant
  $C = C (s)$ such that
  \[ \| f g \|_{B^s_{p, q}} \leqslant C (\| f \|_{L^{\infty}} \| g \|_{B^s_{p,
     q}} + \| f \|_{B^s_{p, q}} \| g \|_{L^{\infty}}) \quad \forall \, f, g
     \in B^s_{p, q} \cap L^{\infty} . \]
\end{proposition}

Another key property of Besov spaces is that they include several other
classical function spaces:
\begin{itemize}
  \item for $s \in \mathbb{R}$, $B^s_{2, 2} (\mathbb{T}^d)$ coincide with the
  Sobolev spaces $H^s (\mathbb{T}^d)$, with equivalent norms;
  
  \item for $s \in (0, 1)$, $B^s_{\infty, \infty} (\mathbb{T}^d)$ coincide
  with $C^s (\mathbb{T}^d)$, the space of periodic $s$-H\"{o}lder continuous
  functions (w.r.t. the canonical distance $d_{\mathbb{T}^d}$), with
  equivalent norms;
  
  \item for $s \in (0, 1)$ and $p \in [1, \infty)$, the spaces $B^s_{p,
  \infty} (\mathbb{T}^d)$, often referred to as Besov--Nikolskii spaces, can
  be characterized by the equivalent norm
  \begin{equation}\label{eq:besov-equiv-norm-1}
  	\|f \tilde{\|}_{B^s_{p, \infty}} \assign \| f \|_{L^p} + \sup_{x \neq y
     \in \mathbb{T}^d} \frac{\left\| f \left( \cdot \, + x \right) - f \left(
     \cdot \, + y \right) \right\|_{L^p}}{d_{\mathbb{T}^d} (x, y)^s}
  \end{equation}
  \item for $s \in (0, 1)$, $p, q \in [1, \infty)$ the space $B^s_{p, q}
  (\mathbb{T}^d)$ has equivalent norm
  \begin{equation}\label{eq:besov-equiv-norm-2}
  \|f \tilde{\|}_{B^s_{p, q}} \assign \| f \|_{L^p} + \int_{\mathbb{T}^d}
     \left( \frac{\left\| f \left( \cdot \, + x \right) - f (\cdot)
     \right\|_{L^p}}{d_{\mathbb{T}^d} (x, 0)^s} \right)^q 
     \frac{1}{d_{\mathbb{T}^d} (x, 0)^d} \mathd x.
  \end{equation}
\end{itemize}

We conclude this appendix by proving some interpolation inequalities, which
played a fundamental role in the proofs in
Section~\ref{sec:lower-bound-inviscid}.

\begin{lemma}
  \label{lem:interpolation-besov}Let $p \in [1, \infty]$, $s_1, s_2 \in
  \mathbb{R}$ with $s_1 < s_2$ and $\theta \in (0, 1)$. Then there exists a
  constant $C = C (p, s_2 - s_1, \theta)$ such that
  \begin{equation}
    \| f \|_{B^{\theta s_1 + (1 - \theta) s_2}_{p, 1}} \leqslant C \| f
    \|_{B^{s_1}_{p, \infty}}^{\theta}  \| f \|^{1 - \theta}_{B^{s_2}_{p,
    \infty}}  \quad \forall \, f \in B^{s_2}_{p, \infty} .
    \label{eq:interpolation-besov}
  \end{equation}
\end{lemma}

\begin{proof}
  The result is well known, see Theorem 2.80 from~{\cite{bahouri}} for the
  statement on $\mathbb{R}^d$; let us provide a self-contained proof.
  
  We may assume $\| f \|_{B^{s_2}_{p, \infty}} = 1$; for any $N \geqslant 0$
  it holds
  \begin{align*}
    \| f \|_{B^{s_{\theta}}_{p, 1}} & = \sum_{j < N} 2^{j (\theta s_1 + (1 -
    \theta) s_2)} \| \Delta_j f \|_{L^p} + \sum_{j \geqslant N} 2^{j (\theta
    s_1 + (1 - \theta) s_2)} \| \Delta_j f \|_{L^p}\\
    & \leqslant \| f \|_{B^{s_1}_{p, \infty}} \sum_{j < N} 2^{j (1 - \theta)
    (s_2 - s_1)} + \| f \|_{B^{s_2}_{p, \infty}} \sum_{j \geqslant N} 2^{- j
    \theta (s_2 - s_1)}\\
    & \lesssim \| f \|_{B^{s_1}_{p, \infty}} 2^{N (1 - \theta) (s_2 - s_1)} +
    2^{- N \theta (s_2 - s_1)} .
  \end{align*}
  Choosing $N$ such that $\| f \|_{B^{s_1}_{p, \infty}} \sim 2^{- N (s_2 -
  s_1)}$ the conclusion then follows.
\end{proof}

\begin{corollary}
  \label{cor:interpolation-besov}For any $s_1, s_2 > 0$ there exists a
  constant $C (s_1, s_2)$ such that
  \begin{equation}
    \| f \|_{L^2} \leqslant C \| f \|_{H^{- s_1}}^{s_2 / (s_1 + s_2)}  \| f
    \|_{B^{s_2}_{2, \infty}}^{s_1 / (s_1 + s_2)} \quad \forall \, f \in
    B^{s_2}_{2, \infty} . \label{eq:interpolation-besov-2}
  \end{equation}
\end{corollary}

\begin{proof}
  Applying Lemma~\ref{lem:interpolation-besov} for the choice $p = 2$, $\theta
  = s_2 / (s_1 - s_2)$ and using Besov embeddings we find
  \[ \| f \|_{L^2} \leqslant \| f \|_{B^0_{2, 1}} \lesssim \| f \|_{B^{-
     s_1}_{2, \infty}}^{\theta}  \| f \|^{1 - \theta}_{B^{s_2}_{2, \infty}}
     \lesssim \| f \|_{H^{- s_1}}^{\theta} \| f \|^{1 - \theta}_{B^{s_2}_{2,
     \infty}} . \qedhere \]
\end{proof}

\section{A simple extension of a result by Wei}\label{app:wei-extension}

Theorem~5.1 from~{\cite{wei}} requires the restriction to work with $u \in C
(\mathbb{T})$, but we show here that such a restriction is not necessary and
in fact the result holds for any $u \in L^1 (\mathbb{T})$, as stated in
Theorem~\ref{thm:wei}. Let us recall the setting: we are interested in the
decay of solutions to complex valued PDEs of the form
\begin{equation}
  \partial_t f + i u f = \nu \partial_y^2 f. \label{eq:pde-app}
\end{equation}
Equation~\ref{eq:pde-app} is well-posed (in the weak sense) for any $u \in L^1
(\mathbb{T})$ and $f_0 \in L^2 (\mathbb{T})$. Indeed, for smooth $u$, any
solution $f$ to~\eqref{eq:pde-app} satisfies
\[ \partial_t \| f \|_{L^2}^2 + 2 \nu \| \partial_y f \|_{L^2}^2 = 0, \]
thus implying that it belongs to $L^2 (0, T ; H^1 (\mathbb{T}))
\hookrightarrow L^2 (0, T ; C (\mathbb{T}))$; therefore we have uniform
estimates for $iu f \in L^2 (0, T ; L^1 (\mathbb{T}))$ only depending on $\| u
\|_{L^1}$. Arguing by weak compactness one can then easily construct weak
solutions to~\eqref{eq:pde-app} for any $u \in L^1 (\mathbb{T})$, establish
their uniqueness, and show that they are the strong limit in $C ([0, T] ; L^2
(\mathbb{T}))$ of those to smooth $u$. Overall, this defines the semigroup $t
\mapsto e^{t (\nu \partial_y^2 - i u)}$ on $L^2 (\mathbb{T})$ for any $u \in
L^1 (\mathbb{T})$ and $\nu > 0$.

\

Identifying $u \in L^1 (\mathbb{T})$ with a $2 \pi$-periodic function, its
primitive $\psi$ is a (non periodic) element of $C (\mathbb{R})$, well defined
up to additive constant; for given $\delta \in (0, 1)$, define
\[ \omega_1 (\delta, u) : = \inf_{x, c_1, c_2 \in \mathbb{R}} \int_{x -
   \delta}^{x + \delta} | \psi (y) - c_1 - c_2 \delta |^2 \mathd y. \]
Denote by $F : \mathbb{R}_{\geqslant 0} \rightarrow [0, \pi / 2]$ the inverse
of $x \mapsto 36 x \tan x$, which is a one-to-one increasing function. The
next statement summarizes some of the main findings from~{\cite{wei}}.

\begin{lemma}
  \label{lem:wei-summary}Let $u \in C (\mathbb{T})$ and $\nu > 0$ be fixed;
  then for all $\delta \in (0, 1)$ and $t \geqslant 0$ it holds
  \begin{equation}
    \| e^{t (\partial_y^2 - i u)} \|_{L^2 \rightarrow L^2} \leqslant \exp
    \left( \frac{\pi}{2} - t \nu \delta^{- 2} F (\delta \nu^{- 2} (\omega_1
    (\delta, u)))^2 \right) . \label{eq:wei-summary-app}
  \end{equation}
\end{lemma}

\begin{proof}
  By time rescaling, the solution $f$ to~\eqref{eq:pde-app} is given by $f (t,
  y) = f^{\nu} (t \nu, y)$ where $f^{\nu}$ solves $\partial_t f^{\nu} + i
  u^{\nu} f^{\nu} = \partial_y^2 f^{\nu}$ where $u^{\nu} = u / \nu$; applying
  the Gearhart--Pr\"{u}ss theorem (Theorem~1.3 from~{\cite{wei}}) to
  $f^{\nu}$, it holds
  \[ \| f_t \|_{L^2} \leqslant \exp \left( \frac{\pi}{2} - t \nu \psi_1
     (u^{\nu}) \right) \]
  where $\psi_1 (u)$ is defined as in Section~4 from~{\cite{wei}}. By
  Lemma~4.3 therein and the 2-homogeneity of $u \mapsto \omega (\delta, u)$,
  for any $\delta \in (0, 1)$ it holds
  \[ \psi_1 (u^{\nu}) \geqslant \delta^{- 2} F (\delta (\omega_1 (\delta,
     u^{\nu})))^2 = \delta^{- 2} F (\delta \nu^{- 2} (\omega_1 (\delta, u)))
  \]
  which gives the conclusion.
\end{proof}

We can now give the

\begin{proof}[Proof of Theorem~\ref{thm:wei}]
  By time rescaling, we can restrict to the case $k = 1$. Now let $u \in L^1
  (\mathbb{T})$ be a function satisfying $\Gamma_{\alpha} (u) > 0$ for some
  $\alpha \in (0, 1)$ and consider a family $\{ u^{\varepsilon}, \varepsilon >
  0 \}$ of continuous functions satisfying $\| u^{\varepsilon} - u \|_{L^1}
  \leqslant \varepsilon$. Denote by $\psi^{\varepsilon}$ the primitive of
  $u^{\varepsilon}$; by the basic inequality $a^2 \geqslant b^2 / 2 - (a -
  b)^2$, for any $\delta (0, 1)$ it holds
  \[ \begin{array}{ll}
       \omega_1 (\delta, u^{\varepsilon}) & \geqslant \frac{1}{2} \omega_1
       (\delta, u) - \sup_{x \in \mathbb{R}} \int_{x - \delta}^{x + \delta} |
       \psi^{\varepsilon} (y) - \psi (y) |^2 \mathd y\\
       & \geqslant \frac{1}{2} \omega_1 (\delta, u) - 2 \delta \| u -
       u^{\varepsilon} \|_{L^1 (\mathbb{T})}^2 .
     \end{array} \]
  Combined with the fact that by definition $\omega_1 (\delta, u) \geqslant
  2^{2 \alpha + 3} \delta^{2 \alpha + 3} \Gamma_{\alpha} (u)^2$, we deduce
  \begin{equation}
    \omega_1 (\delta, u^{\varepsilon}) \geqslant 2^{2 \alpha + 2} \delta^{2
    \alpha + 3} \Gamma_{\alpha} (u)^2 - 2 \delta \varepsilon^2 \quad \forall
    \, \delta \in (0, 1), \, \varepsilon > 0.
  \end{equation}
  Now fix $\nu > 0$ and define $C_1 = e^{\pi / 2}$, $C_2 = 2^{2 \alpha + 2}
  \Gamma_{\alpha} (u)^2$; applying Lemma~\ref{lem:wei-summary} to
  $u^{\varepsilon}$, exploiting the fact that $F$ is increasing, and choosing
  $\delta = \nu^{1 / (\alpha + 2)}$, we obtain
  \begin{equation}
    \| e^{t (\partial_y^2 - i u^{\varepsilon})} \|_{L^2 \rightarrow L^2}
    \leqslant C_1 \exp \left( - t \nu^{\frac{\alpha}{\alpha + 2}} F \left( C_2
    - 2 \nu^{- 2 \frac{\alpha + 1}{\alpha + 2}} \varepsilon^2 \right)^2
    \right) \quad \forall \, t \geqslant 0 \label{eq:app-proof2}
  \end{equation}
  where the estimate holds for all $\varepsilon > 0$ small enough such that
  $C_2 - 2 \nu^{- 2 (\alpha + 1) / (\alpha + 2)} \varepsilon^2 > 0$.
  
  Since the semigroup $e^{t (\nu \partial_y^2 - i u^{\varepsilon})}$ pointwise
  converges to $e^{t (\nu \partial_y^2 - i u^{\varepsilon})}$ as $\varepsilon
  \rightarrow 0^+$, passing to the limit on both sides
  of~\eqref{eq:app-proof2} gives the conclusion.
\end{proof}

\

\tmtextbf{Declarations.}
The authors have no relevant financial or non-financial interests to disclose.

\

\end{document}